\documentclass[reqno,english]{amsart}
\usepackage[utf8]{inputenc}
\usepackage{xcolor}
\usepackage{babel}
\usepackage{url}
\usepackage{enumitem}
\usepackage{amstext}
\usepackage{amsthm}
\usepackage{amssymb}
\usepackage{esint}
\usepackage[all]{xy}
\PassOptionsToPackage{normalem}{ulem}
\usepackage{ulem}
\usepackage[unicode=true,pdfusetitle,
 bookmarks=true,bookmarksnumbered=false,bookmarksopen=false,
 breaklinks=false,pdfborder={0 0 0},pdfborderstyle={},backref=false,colorlinks=true]
 {hyperref}
\hypersetup{
 pdfborderstyle=,pdfborderstyle={}}

\makeatletter

\providecolor{lyxadded}{rgb}{0,0,1}
\providecolor{lyxdeleted}{rgb}{1,0,0}
\DeclareRobustCommand{\lyxsout}[1]{\ifx\\#1\else\sout{#1}\fi}

\numberwithin{equation}{section}
\numberwithin{figure}{section}
\theoremstyle{plain}
\newtheorem{thm}{\protect\theoremname}
\theoremstyle{definition}
\newtheorem{defn}[thm]{\protect\definitionname}
\theoremstyle{definition}
\newtheorem{example}[thm]{\protect\examplename}
\theoremstyle{remark}
\newtheorem{rem}[thm]{\protect\remarkname}
\theoremstyle{plain}
\newtheorem{cor}[thm]{\protect\corollaryname}
\theoremstyle{plain}
\newtheorem{lem}[thm]{\protect\lemmaname}

\usepackage{pdfsync}
\usepackage{graphics}
\usepackage{epstopdf}
\usepackage[all]{xy}
\usepackage{amscd}
\usepackage{enumitem}
\usepackage{mathtools}
\usepackage{bbold}
\usepackage[utf8]{inputenc}
\setlist[enumerate]{leftmargin=*,label=(\roman*),align=left}

\makeatletter
\@namedef{subjclassname@2020}{%
  \textup{2020} Mathematics Subject Classification}
\makeatother

\newcommand{\xyR}[1]{ \makeatletter
\xydef@\xymatrixrowsep@{#1} \makeatother} % end of \xyR
\newcommand{\xyC}[1]{ \makeatletter
\xydef@\xymatrixcolsep@{#1} \makeatother} % end of \xyC
\entrymodifiers={++[ ][F]} % to have more space around objects in diagrams

\newcommand{\ra}{\longrightarrow}
\newcommand{\xra}[1]{\xrightarrow{\ \ #1\ \ }} % extendible arrow
\newcommand{\field}[1]{\mathbb{#1}}
\newcommand{\R}{\field{R}} % reals
\newcommand{\N}{\field{N}} % naturals
\newcommand{\eps}{\varepsilon} % for the sake of brevity only
\renewcommand{\phi}{\varphi}
\newcommand{\diff}[1]{\ifmmode\mathchoice{\hbox{\rm d}#1}  % displaystyle
 {\hbox{\rm d}#1}  % normal 
 {\scalebox{0.75}{$\hbox{\rm d}#1$}}  % scriptstyle 
 {\scalebox{0.35}{$\hbox{\rm d}#1$}}  % scriptscriptstyle
 \fi} % dt,dx,... for integrals
\newcommand{\abs}[2][\empty]{\ifx#1\empty\left|#2\right|%
\else#1\vert #2 #1\vert\fi}% optional arg=size
\newcommand{\Coo}{\mbox{\ensuremath{\mathcal{C}}}^{\infty}} % C infinity
\DeclareMathOperator{\Set}{{\bf Set}} % category of Sets
\newcommand{\OR}{{\mathcal{O}}\mbox{$\R^\infty$}} % category of open sets in numerical spaces
\newcommand{\Rtil}{\widetilde \R} % real Colombeau generalized number
\newcommand{\frontRise}[2]{\ifmmode\mathchoice{{\vphantom{#1}}^{\scalebox{0.6}{$#2$}}}  % displaystyle
 {{\vphantom{#1}}^{\scalebox{0.56}{$#2$}}}  % normal 
 {{\vphantom{#1}}^{\scalebox{0.47}{$#2$}}}  % scriptstyle 
 {{\vphantom{#1}}^{\scalebox{0.35}{$#2$}}}\fi} % scriptscriptstyle 
\newcommand{\RC}[1]{\frontRise{\R}{#1}\Rtil}
\newcommand{\frontRiseDown}[3]{\ifmmode\mathchoice{{\vphantom{#1}}^{\scalebox{0.6}{$#2$}}_{\scalebox{0.6}{$#3$}}}  % displaystyle
 {{\vphantom{#1}}^{\scalebox{0.56}{$#2$}}_{\scalebox{0.56}{$#3$}}}  % normal 
 {{\vphantom{#1}}^{\scalebox{0.47}{$#2$}}_{\scalebox{0.47}{$#3$}}}  % scriptstyle 
 {{\vphantom{#1}}^{\scalebox{0.35}{$#2$}}_{\scalebox{0.35}{$#3$}}}\fi} % scriptscriptstyle 

\newcommand{\rcrho}{\RC{\rho}}
\newcommand{\rti}{\RC{\rho}}
\newcommand{\gsf}{\frontRise{\mathcal{G}}{\rho}\mathcal{GC}^{\infty}}

\newcommand{\rhoGs}{\frontRise{\mathcal{G}}{\rho}\mathcal{G}^{\text{s}}}
\newcommand{\rhoMod}{\frontRise{\mathcal{E}}{\rho}\mathcal{E}_{\text{M}}}
\newcommand{\rhoNeg}{\frontRise{\mathcal{N}}{\rho}\mathcal{N}}

\newcommand{\ptind}{\displaystyle \mathop {\ldots\ldots\,}} % marks with smth over. Usage: \ptind^{...}
\newcommand{\DIff}{ \quad\;\; :\!\iff \quad } % :iff by definition
\makeatother

\providecommand{\corollaryname}{Corollary}
\providecommand{\definitionname}{Definition}
\providecommand{\examplename}{Example}
\providecommand{\lemmaname}{Lemma}
\providecommand{\remarkname}{Remark}
\providecommand{\theoremname}{Theorem}

\begin{document}
\title{Universal properties of spaces of generalized functions}
\author{Djamel eddine Kebiche \and Paolo Giordano}
\thanks{P.~Giordano has been supported by grants P33538, P34113, P33945,
DOI: 10.0.217.224/PAT1996423, of the Austrian Science Fund FWF}
\address{\textsc{Faculty of Mathematics, University of Vienna, Austria, Oskar-Morgenstern-Platz
1, 1090 Wien, Austria}}
\thanks{D.~Kebiche has been supported by grant P33538 of the Austrian Science
Fund FWF}
\email{\texttt{paolo.giordano@univie.ac.at, djameleddine.kebiche@univie.ac.at}}
\begin{abstract}
By means of several examples, we motivate that universal properties
are the simplest way to solve a given mathematical problem, explaining
in this way why they appear everywhere in mathematics. In particular,
we present the co-universal property of Schwartz distributions, as
the simplest way to have derivatives of continuous functions, Colombeau
algebra as the simplest quotient algebra where representatives of
zero are infinitesimal, and generalized smooth functions as the universal
way to associate set-theoretical maps of non-Archimedean numbers defined
by nets of smooth functions (e.g.~regularizations of distributions)
and having arbitrary derivatives. Each one of these properties yields
a characterization up to isomorphisms of the corresponding space.
The paper requires only the notions of category, functor, natural
transformation and Schwartz's distributions, and introduces the notion
of universal solution using a simple and non-abstract language.
\end{abstract}

\maketitle

\section{Introduction}

Mathematicians eventually try to solve a problem in the best possible
way. For example, we can consider a geometrically intrinsic solution,
or the best computational algorithm, or the most general answer. Frequently
motivated by the searching of beauty, \cite{Had49}, we can also require
that the solution is the ``simplest'' one, i.e.~it has to use the
minimal amount of conventional constructions and data other than the
given ones from which the problem must depend on. At a first sight,
a non-trivial possible mathematical formalization of the idea of \emph{simplest
solution} should involve information theory (see e.g.~\cite{Sch07}
and references therein) or mathematical logic. In this paper, using
only a minimal amount of category theory, we see a common informal
interpretation of \emph{universal solution} as the simplest way to
solve a given problem. It is well-known that universal constructions
appear everywhere in mathematics, \cite{Lei14}, and hence this interpretation
justifies why this happens. We list several examples justifying this
interpretation, in particular for spaces of generalized functions
(GF) both in linear and nonlinear frameworks.

We will see that a universal solution not only candidates itself as
the simplest way to solve a given problem, but its universal property
is able to highlight what are the data of the problem and the conventional
choices in any other possible construction. Frequently, this paves
the way for generalizations, and it always directly yields an axiomatic
characterization of these universal solutions. In the point of view
of several mathematicians, universal properties are so important that
they take them as a starting point: ``it is not important how you
solve this problem, because the key point is that you have a universal
solution, which is unique up to isomorphisms''.

Concerning spaces of GF, in this paper we clearly consider Schwartz's
distributions as ``the simplest way to have derivatives of continuous
functions'' (see \cite{Sch50}), and hence we show the corresponding
(not-well known) universal property in Sec.~\ref{sec:distributions}
(see also \cite{MeMu00}, but where the universal property mistakenly
lacks condition Thm.~\ref{thm:UP of D'}.\ref{enu:distPoly}). Following
the algebraic construction of Sebastiao e Silva (\cite{SeS61}) of
distributions, we see how to obtain a similar universal construction
for distributions on Hilbert spaces. Any other solution of the same
problem will reasonably satisfy the same minimal and meaningful universal
property, and hence it will be isomorphic to our solution, see Sec.~\ref{subsec:Application}.

Among nonlinear settings for GF extending Schwartz's distributions,
Co\-lom\-beau's special algebra (see e.g.~\cite{Col84,Col85,Col92,GKOS})
is frequently perceived as the simplest one. In Sec.~\ref{sec:Colombeau},
we prove that it is actually the most simple quotient algebra. We
will also consider the recent \emph{generalized smooth functions}
(see \cite{GKV15,MTG,GKV24} and references therein) because it is
even more general than Colombeau's algebras, but with several improved
properties such as more general domains, the closure with respect
to composition, a better integration theory and Hadamard's well-posedness
for every PDE (in infinitesimal neighborhoods), see Sec.~\ref{sec:GSF}.
As a secondary result, we hence have an axiomatic description up to
isomorphisms of Colombeau's special algebra and of generalized smooth
functions. In particular, the ring of Colombeau's generalized numbers
reveals to be the simplest quotient ring $\mathcal{M}/\sim$ containing
the infinite numbers $[\eps^{-q}]_{\sim}$, $q\in\N,$ and where every
zero-net $[x_{\eps}]_{\sim}=0$ is generated by an infinitesimal function:
$\lim_{\eps\to0^{+}}x_{\eps}=0$ (see Sec.~\ref{subsec:A-particular-case};
see also \cite{Tod09} for another axiomatic description in the framework
of nonstandard analysis and where the latter property does not hold).

In the following, we will use the conventions:
\begin{itemize}
\item universal = terminal = limit = projective: the unique arrow arrives
to the universal object.
\item co-universal = initial = co-limit = injective: the unique arrow starts
from the universal object.
\end{itemize}
\noindent The paper is self-contained, in the sense that it requires
only the notions of category, functor, natural transformation and
Schwartz's distributions.

\noindent We start by introducing the notion of universal solution
using a simple and non-abstract language.

\section{General definition of (co-) universal property}

We start by defining in general terms what a universal property is.
We will use only basic notions of category theory, and will give a
definition near to the common use of universal properties.
\begin{defn}
\textcolor{black}{\label{def:Universal}Let $\mathbf{C}$ be a category.
Let $\mathcal{P}(C)$ and $\mathcal{{Q}}(f,A,B)$ be two properties
of $A$, $B$, $C$ and $f$, where $A$, $B$, $C$ are objects of
$\mathbf{C}$ and $f$ is an arrow of $\mathbf{C}$. Assume that:
\begin{equation}
\mathcal{{Q}}(f,A,B),\:\mathcal{{Q}}(g,B,C)\ \Rightarrow\ \mathcal{{Q}}(g\circ f,A,C),\label{eq:comp}
\end{equation}
\begin{equation}
\mathcal{{Q}}(1_{A},A,A).\label{eq:identity}
\end{equation}
Then we say that $C$ }\textcolor{black}{\emph{is a universal solution
of}}\textcolor{black}{{} $\mathcal{{P}}$ }\textcolor{black}{\emph{with
respect to}}\textcolor{black}{{} $\mathcal{Q}$ if}
\begin{enumerate}
\item \textcolor{black}{\label{enu:PcUniversal}$\mathcal{P}(C)$, i.e.~the
object $C$ solves the problem $\mathcal{P}(-)$.}
\item \textcolor{black}{\label{enu:uniqueArrUniversal}$\forall D\in\mathbf{C}:\ \mathcal{P}(D)\ \Rightarrow\ \exists!\varphi:D\longrightarrow C:\ \mathcal{Q}(\varphi,D,C)$,
i.e.~for any other solution $D$ of the same problem $\mathcal{P}(-)$,
we can find one and only one morphism $\phi:D\ra C$ that satisfies
the property $\mathcal{Q}$.}
\end{enumerate}
\textcolor{black}{Similarly, we say that $C$ }\textcolor{black}{\emph{is
a co-universal solution of}}\textcolor{black}{{} $\mathcal{{P}}$ }\textcolor{black}{\emph{with
respect to}}\textcolor{black}{{} $\mathcal{Q}$ if}
\begin{enumerate}[resume]
\item \textcolor{black}{$\mathcal{P}(C)$,}
\item \textcolor{black}{$\forall D\in\mathbf{C}:\ \mathcal{P}(D)\ \Rightarrow\ \exists!\varphi:C\longrightarrow D:\ \mathcal{Q}(\varphi,C,D)$.}
\end{enumerate}
\end{defn}

The proof of the following theorem trivially generalizes the classical
proofs concerning the uniqueness of universal objects up to isomorphisms:
\begin{thm}
\label{thm:UPuniqueUpToIso}\textcolor{black}{Suppose that $C_{1}$
and $C_{2}$ are two (co-)universal solutions of $\mathcal{P}$ with
respect to $\mathcal{Q}$. Then $C_{1}$ is isomorphic to $C_{2}$
in $\mathbf{C}$.}
\end{thm}

\begin{proof}
\textcolor{black}{Since $C_{1}$ is a universal solution of $\mathcal{P}$
with respect to $\mathcal{Q}$, using \ref{enu:uniqueArrUniversal}
of Def.~\ref{def:Universal} for $D=C_{2}$, there exists a unique
$\varphi_{1}:C_{2}\longrightarrow C_{1}$ such that the property $\mathcal{Q}(\varphi_{1},C_{2},C_{1})$
holds. In a similar way, there exists a unique $\varphi_{2}$ such
that $\varphi_{2}:C_{1}\longrightarrow C_{2}$ so we have $\mathcal{Q}(\varphi_{2},C_{1},C_{2})$.
By assumption \eqref{eq:comp} on $\mathcal{{Q}}$, the property $\mathcal{Q}(\varphi_{2}\circ\varphi_{1},C_{2},C_{2})$
holds. Using again Def.~\ref{def:Universal}\ref{enu:uniqueArrUniversal}
with $D=C_{2}$, we get that only one arrow $\varphi$ satisfies $\mathcal{Q}(\varphi,C_{2},C_{2})$.
Since $\mathcal{Q}(1_{C_{2}},C_{2},C_{2})$ also holds by \eqref{eq:identity},
then $\varphi_{2}\circ\varphi_{1}=1_{C_{2}}$. In a similar way, we
have $\varphi_{1}\circ\varphi_{2}=1_{C_{1}}$, which proves the theorem.}
\end{proof}
\textcolor{black}{Starting from the properties $\mathcal{P}$ and
$\mathcal{Q}$, we can define a new category $\mathbf{C}(\mathcal{P},\mathcal{Q})$.
Its objects are the objects of the category $\mathbf{C}$ that satisfy
the property $\mathcal{{P}}$ (i.e.~all the solutions of our problem
$\mathcal{P}(-)$), and its arrows are the arrows $\phi$ of the category
$\mathbf{C}$ such that $\mathcal{{Q}}(f,C,D)$ holds (so that the
property $\mathcal{Q}$ links all these solutions), i.e.:}
\begin{itemize}
\item \textcolor{black}{$C\in\mathbf{C}(\mathcal{P},\mathcal{Q}):\iff\mathcal{{P}}(C)$,}
\item \textcolor{black}{$D\stackrel{\varphi}{\longrightarrow}C$ in $\mathbf{C}(\mathcal{P},\mathcal{Q})$$:\iff\mathcal{{Q}}(\varphi,D,C)$,
$D\stackrel{\varphi}{\longrightarrow}C$ in $\mathcal{\mathbf{C}}$,}
\item \textcolor{black}{$\theta=\psi\circ\varphi$ in $\mathbf{C}(\mathcal{{P}},\mathcal{{Q}}):\iff$$\theta=\psi\circ\varphi$
in $\mathbf{C}.$}
\end{itemize}
\textcolor{black}{Then, we have that $C$ is a universal solution
of $\mathcal{{P}}$ with respect to $\mathcal{{Q}}$ if and only if
$C$ is terminal in $\mathbf{C}(\mathcal{{P}},\mathcal{{Q}})$ (i.e.~for
all $D\in\mathbf{C}(\mathcal{P},\mathcal{Q})$ there exists one and
only one $\varphi:D\ra C$ in $\mathbf{C}(\mathcal{P},\mathcal{Q})$),
and $C$ is a co-universal solution of $\mathcal{{P}}$ with respect
to $\mathcal{{Q}}$ if and only if $C$ is initial in $\mathbf{C}(\mathcal{{P}},\mathcal{{Q}})$
(i.e.~for all $D\in\mathbf{C}(\mathcal{P},\mathcal{Q})$ there exists
one and only one $\varphi:C\ra D$ in $\mathbf{C}(\mathcal{P},\mathcal{Q})$).}

As we mentioned above, a (co-)universal solution of $\mathcal{P}$
is considered as the (co-)simplest or (co-)most natural solution of
that problem. Even considering only the following elementary examples,
we can start to justify this interpretation:
\begin{example}
\ 

\begin{enumerate}
\item Let's consider the problem to specify a topology on a set $X\in\mathbf{Set}$.
The category $\mathcal{\mathbf{C}}$ in this example is the category
of all the topologies on $X$ viewed as a poset, i.e.~``$\subseteq$''
is the unique arrow of $\mathcal{\mathbf{C}}$, and we write $\tau\subseteq\sigma$
if the topology $\tau$ is coarser than the topology $\sigma$. The
properties $\mathcal{{P}}$ and $\mathcal{{Q}}$ are defined as follow.
\begin{align*}
\mathcal{P}(\tau)\  & :\iff\ \tau\;\text{is a topology on }X,\\
\mathcal{Q}(i,\tau,\sigma)\  & :\iff\ i=\subseteq,\ \tau\subseteq\sigma.
\end{align*}
The trivial topology $(\left\{ \emptyset\right\} ,X)$ is the co-universal
solution of the property $\mathcal{P}$ with respect to the property
$\mathcal{Q}$ and the discrete topology is the universal solution.
Clearly, these also appear as trivial solutions; on the other hand,
note that they are also the simplest/non-conventional solutions starting
from the unique data $X\in\mathbf{Set}$ and with respect to the problem
``set a topology on $X$'': any other solution would necessarily
introduce (in the case of the trivial topology) or delete (in the
case of the discrete topology) something which is not related to the
problem or the data itself. This example also shows that the notion
of \emph{simplest solution} can be implemented in two ways: from ``below''
(co-universal) or from ``above'' (universal).
\item \textcolor{black}{Let $R$ be a ring and let $x\not\in R$. What would
be the smallest/simplest ring containing both $x$ and $R$? Any ring
that contains $x$ and $R$ must contain also sums of terms of the
form $r\cdot x^{n}$ for any integer $n$ and any element $r\in R$.
Intuitively, the simplest solution is therefore the ring of polynomials
$R[x].$ The co-universal property can be highlighted as follow: Let
$S\in\mathbf{Ring}$ be a ring, then we can consider the property
$\mathcal{P}(S,s)$ whenever $x\in S$ and $s:R\longrightarrow S$
is a ring homomorphism, and the property $\mathcal{Q}(f,\left(S,s\right),\left(L,l\right))$
if $S\stackrel{f}{\longrightarrow}L$ in $\mathbf{Ring}$ (i.e.~it
is a morphism of rings) and $f\circ s=l$. The ring of polynomials
$R[x]$ is the co-universal solution of $\mathcal{P}$ with respect
to $\mathcal{Q}$, i.e.~the simplest way to extend the ring $R$
by adding a new element $x\notin R$. Clearly, we have $\mathcal{P}(R[x],i)$,
where $i:R\ra R[x]$ is the inclusion. Let $S\in\mathbf{Ring}$, and
let $s:R\ra S$ be a ring homomorphism, i.e.~$\mathcal{P}(S,s)$
holds, then the unique $\phi:R[x]\longrightarrow S$ of Def.~\ref{def:Universal}.\ref{enu:uniqueArrUniversal}
is given by $\phi\left(\sum_{i}r_{i}x^{i}\right)=\sum_{i}s(r_{i})x^{i}$.}
\item Let $(X,d)$ be a metric space and $(X^{*},d^{*})$ be its completion
as the usual quotient of Cauchy sequences: $(x_{n})_{n}\sim(y_{n})_{n}$
if and only if $\lim_{n\rightarrow\infty}d(x_{n},y_{n})=0$. Define
an isometry $\phi:X\longrightarrow X^{*}$ by setting $\phi(x):=[x]_{\sim}$,
where $[x]_{\sim}$ is the equivalent class generated by the constant
sequence $x_{n}=x\in X$; we have that $\phi(X)$ is dense in $X^{*}$
(see e.g.~\cite{Simm63}). The triple $(X^{*},d^{*},\phi)$ is co-universal
among all the triples $(Y,\delta,\psi)$, where $(Y,\delta)$ is a
complete metric space and $\psi:X\longrightarrow Y$ is an isometry
such that $\psi(X)$ is dense in $Y$. There is therefore a unique
map $\iota:X^{*}\ra Y$ such that $\iota\circ\phi=\psi$, which is
defined as follows: Let $x^{*}\in X^{*}$. Since $\phi(X)$ is dense
in $X^{*}$, there exists a sequence $(x_{n})_{n}$ of $X$ such that
$(\phi(x_{n}))_{n}$ converges to $x^{*}$. The sequence $(\phi(x_{n}))_{n}$
is a Cauchy sequence and since $\phi$ and $\psi$ are isometries,
the sequence $(\psi(x_{n}))_{n}$ is also a Cauchy sequence in $Y$
which converges because $Y$ is Cauchy complete. We can thus set $\text{\ensuremath{\iota}}(x^{*}):=\lim_{n\rightarrow\infty}(\psi(x_{n}))$
which is well defined because $\phi$ and $\psi$ are isometries.
\item Let $U$, $V\in\textbf{Vect}$ be vector spaces. The simplest way
to obtain a bilinear map $U\times V\xra{b}T$ into another vector
space $T$ is the tensor product $T=U\otimes V$, with $b(u,v)=u\otimes v$.
This construction is indeed co-universal with respect to the properties:
\begin{align*}
\mathcal{P}(b,T)\DIff & T\in\textbf{Vect},\ U\times V\xra{b}T\text{ is bilinear}\\
\mathcal{Q}(\phi,(b,T),(w,W))\DIff & \phi:T\ra W\text{ in }\textbf{Vect}.
\end{align*}
It is well-known that there are several ways to define the tensor
product $U\otimes V$, even if they all satisfy this co-universal
property (and hence, by Thm.~\ref{thm:UPuniqueUpToIso}, they are
all isomorphic as vector spaces). Note that the category $\mathbf{C}$
of Def.~\ref{def:Universal} is the category of all the pairs $(b,T)$
satisfying $\mathcal{P}(b,T)$.
\end{enumerate}
\end{example}

We underscore that in all these universal solutions (as well as in
products, sums, quotients, etc.~of spaces) there are no conventional
choices and they are the most natural solutions: any other (non-isomorphic)
solution would appear as less natural, e.g.~by adding (to the co-universal
solution) or subtracting (to the universal solution) anything that
does not strictly depend on the data of the problem.

\subsection{Preliminary notions: presheaf and sheaf}

For the sake of completeness, and also to specify all our notations,
in this section we briefly recall the notions of presheaf and sheaf,
because they are used in our universal characterization of spaces
of GF.

In the following, we denote by $\mathbf{Set}$ the category of sets
and functions, by $\textbf{Mod}_{R}$ the category of modules over
the ring $R$, so that $\mathbf{Vect}_{K}:=\textbf{Mod}_{K}$ is the
category of vector spaces over a given field $K$, $\OR$ is the category
having as objects open sets $U\subseteq\R^{u}$ of any dimension $u\in\N=\left\{ 0,1,2,\ldots\right\} $,
and smooth functions as arrows, and finally $\mathbf{Ring}$ is the
category of rings and \textcolor{black}{ring-homomorphisms. If $\mathbb{T}=\left(|\mathbb{T}|,\tau\right)$
is a topological space, we use the same symbol to also denote the
category induced by its open sets as a preorder, i.e.~}the category
of open sets $A\in\tau$ of the given topology and only one arrow
``$\subseteq$'', i.e.~we write $A\overset{}{\xra{\subseteq}}B$
in $\mathbb{T}$ if $A\subseteq B$. We finally denote by $\mathbf{C}^{\text{op}}$
the opposite of any category $\mathbf{C}$; for example, we write
$f\in(\OR)^{\text{op}}(A,B)$ if $f\in\mathcal{C}^{\infty}(B,A)$
is a smooth function from $B\subseteq\R^{b}$ into $A\subseteq\R^{a}$,
and $\mathbb{T}^{\text{op}}(A,B)$ is non empty if and only if $B\subseteq A$.
\begin{defn}
\label{def:sheaf}~
\begin{enumerate}
\item Let $R$ be a ring. A presheaf $P$ of $\textbf{Mod}_{R}$ is a functor
$P:\mathbb{T}^{\text{op}}\longrightarrow$$\textbf{Mod}_{R}$. We
denote by $P(U)\in\textbf{Mod}_{R}$ its evaluation at $U\in\mathbb{T}^{\text{op}}$
and by $P_{U,V}:=P(U\leq V):P(U)\longrightarrow P(V)$ its evaluation
on the arrow $U\supseteq V$. The map $P_{U,V}$ is called \emph{restriction}
from $U$ to $V$.
\item If $(U_{j})_{j\in J}$ is a covering in $\mathbb{T}$ of $U\in\mathbb{\mathbb{T}^{\text{op}}},$
then we say that $(f_{j})_{j\in J}$ is a \emph{$P$-compatible family}
if and only if
\begin{enumerate}
\item $\forall j\in J:\ f_{j}\in P(U_{j})$.
\item $\forall j,h\in J:\ P_{U_{j},U_{j}\cap U_{h}}(f_{j})=P_{U_{h},U_{h}\cap U_{j}}(f_{h})$.
\end{enumerate}
\item \label{enu:sheaf }Moreover, we say that $P:\mathbb{T}^{\text{op}}\longrightarrow\textbf{Mod}_{R}$
is a \emph{sheaf} if it is a presheaf satisfying the following conditions;
for any $U\in\mathbb{T}^{\text{op}}$, for any covering $(U_{j})_{j\in J}$
of $U$ in $\mathbb{T}$ and for any $P$-compatible family $(f_{j})_{j\in J}$:
\begin{enumerate}
\item \label{enu:sheaf 1}If $f$, $g\in P(U)$ and $P_{U,U_{j}}(f)=P_{U,U_{j}}(g)$
for all $j\in J$, then $f=g$ (\emph{locality condition}); if $P$
satisfies only this condition, it is called a \emph{separated presheaf}
or a \emph{monopresheaf}.
\item \label{enu:sheaf 2}$\exists f\in P(U)\,\forall j\in J:\ P_{U,U_{j}}(f)=f_{j}$
(\emph{gluing condition}).
\end{enumerate}
\item Finally, if $P$, $Q:\mathbb{T}^{\text{op}}\longrightarrow\textbf{Mod}_{R}$
are sheaves, we say that $\phi:P\ra Q$ is a \emph{sheaf morphism}
if $\phi$ is a natural transformation from $P$ to $Q$, i.e.~it
is a family $\left(\phi_{U}\right)_{U\in\mathbb{T}}$ such that $Q_{U,V}\circ\phi_{U}=\phi_{V}\circ P_{U,V}$
in $\textbf{Mod}_{R}$ for all $U$, $V\in\mathbb{T}$ such that $U\supseteq V$.
\end{enumerate}
\end{defn}

\noindent Clearly, conditions \ref{enu:sheaf 1}, \ref{enu:sheaf 2}
imply $\exists!f\in P(U)\,\forall j\in J:\ P_{UU_{j}}(f)=f_{j}$;
we set $P_{U}\left[\left(f_{j}\right)_{j\in J}\right]:=f$ and call
it the $P$-\emph{gluing of the family $\left(f_{j}\right)_{j\in J}$}.
For example, it is not hard to prove (see e.g.~\cite{McMo94}) that
\begin{align}
P_{UV}\left(P_{U}\left[\left(f_{j}\right)_{j\in J}\right]\right) & =P_{V}\left[\left(P_{U_{j},V\cap U_{j}}(f_{j})\right)_{j\in J}\right],\label{eq:glueFam}\\
\psi_{U}\left(P_{U}\left[\left(f_{j}\right)_{j\in J}\right]\right) & =Q_{U}\left[\left(\psi_{U_{j}}(f_{j})\right)_{j\in J}\right]\quad\text{if }\psi:P\ra Q\text{ is a sheaf morphism.}\label{eq:morphGlue}
\end{align}

\section{\label{sec:distributions}Co-universal property of Schwartz distributions}

\textcolor{black}{In this section, we want to show a co-universal
property of the space of Schwartz's distributions. More precisely,
as stated in \cite{Sch50}, in this section we formalize the idea
that the sheaf $\mathcal{D}'$ of Schwartz distributions is the simplest
sheaf where we can take derivatives of continuous functions }\textcolor{black}{\emph{and}}\textcolor{black}{{}
preserving partial derivatives $\partial_{k}f$ of fun}ctions $f$
which are continuously differentiable in the $k$-th variable. A similar
statement can be found in \cite{Hor90}: ``\emph{In differential
calculus one encounters immediately the unpleasant fact that not every
function is differentiable. The purpose of distribution theory is
to remedy this flaw; indeed, the space of distributions is essentially
}the smallest extension\emph{ of the space of continuous functions
where differentiability is always well defined}''. Co-universal properties
correspond to this informal notion of ``smallest extension''. This
formalization also allows us to understand the importance of preservation
of partial derivative of sufficiently regular functions, which is
not explicitly included in the previous statement. For this reason,
we define
\begin{defn}
\label{def:Calpha}Let $U\subseteq\R^{n}$ be an open set, $\alpha\in\N^{n}$
be a multi-index, and $k=1,\ldots,n$, then:
\begin{enumerate}
\item \label{enu:U_i(x)}For all $x\in U$, we set $U_{k}(x):=\left\{ t\in\R\mid(x_{1},\ldots,x_{k-1},t,x_{k+1},\ldots,x_{n})\in U\right\} $.
We hence have a map $j_{k}:t\in U_{k}(x)\mapsto(x_{1},\ldots,x_{k-1},t,x_{k+1},\ldots,x_{n})\in U$.
\item Let $\alpha=(0,\ptind^{k-1},0,1,0,\ldots,0)=:e_{k}$, and $f\in\mathcal{C}^{0}(U)$.
Then, we write $f\in\mathcal{C}^{\alpha}(U)$ if $f$ is of class
$1$ in the $k$-th variable, i.e.~
\[
\forall x\in U:\ f\circ j_{k}\in\mathcal{C}^{1}\left(U_{k}(x)\right),
\]
and $\partial_{k}f:=(f\circ j_{k})'\in\mathcal{C}^{0}(U)$. The space
$\mathcal{C}^{e_{k}}(U)$ is also denoted by $\mathcal{C}_{k}^{1}(U)$.
\item \label{enu:Calpha}If $\alpha\in\N^{n}$, the set of all the functions
of class $\alpha_{k}$ in the $k$-th variable ($k=1,\ldots,n$) is
\[
\mathcal{C}^{\alpha}(U):=\left\{ f\in\mathcal{C}^{0}(U)\mid\forall k=1,\ldots,n:\ \alpha_{k}\neq0\ \Rightarrow\ f\in\mathcal{C}^{e_{k}}\left(U\right),\ \partial_{k}f\in\mathcal{C}^{\alpha-e_{k}}(U)\right\} .
\]
In the usual way, it is possible to prove that $\mathcal{C}^{\alpha}$
is a sheaf. In case $\alpha=je_{k}$, the space $\mathcal{C}^{\alpha}(U)$
is also denoted by $\mathcal{C}_{k}^{j}(U)$. Note that if $f\in\mathcal{C}^{\alpha}(U)$
and $k$, $j$ are such that $\alpha_{k}$, $\alpha_{j}\neq0$, then
by Schwarz's theorem we have $\partial_{k}\partial_{j}f=\partial_{j}\partial_{k}f$
on $U$.
\item We say that \textcolor{black}{$U$ is an $n$-dimensional interval
if $U=(c_{1}-r,c_{1}+r)\times\ptind^{n}\times(c_{n}-r,c_{n}+r)$ for
some $c\in\R^{n}$ and $r\in\R_{>0}$.}
\end{enumerate}
\end{defn}

\noindent In what follows, the notations $\xymatrix{\mathcal{C}_{k}^{1}\ar@{^{(}->}@<2pt>[r]^{\iota_{k}}\ar@<-2pt>[r]_{\partial_{k}} & \mathcal{C}^{0}}
$, are used to denote the inclusion and the partial derivatives of
$\mathcal{C}_{k}^{1}$-functions (thought of as sheaves morphisms,
e.g.~we think $\iota_{kU}:\mathcal{C}_{k}^{1}(U)\hookrightarrow\mathcal{C}^{0}(U)$
as a natural transformation).
\begin{rem}
\label{rem:Schwartz's-solution}Schwartz's solution leads to the following
objects:
\begin{enumerate}
\item $\mathcal{D}':\left(\R^{n}\right)^{\text{op}}\ra\mathbf{Vect}_{\R}$
is the sheaf of real valued distributions on $\R^{n}$.
\item $\mathcal{C}^{0}\xra{\lambda}\mathcal{D}'$ is the inclusion of the
space of continuous functions into the space of distributions. The
map $\lambda$ is a sheaf morphism, i.e.~it is a natural transformation:
$\lambda_{U}:\mathcal{C}^{0}(U)\ra\mathcal{D}'(U)$ for all open sets
$U$, $V\subseteq\R^{n}$ with $V\subseteq U$, such that the following
diagram commutes 
\[
\xymatrix{\mathcal{C}^{0}(U)\ar[d]_{\mathcal{C}_{U,V}^{0}}\ar[r]^{\lambda_{U}} & \mathcal{D}'(U)\ar[d]^{\mathcal{D}'_{U,V}}\\
\mathcal{C}^{0}(V)\ar[r]_{\lambda_{V}} & \mathcal{D}'(V)
}
\]
Therefore, $\mathcal{C}_{U,V}^{0}(f):=f|_{V}$ and $\mathcal{D}'_{U,V}(T):=T|_{V}$
are the corresponding restrictions.
\item \label{enu:compDer}$\mathcal{D}'\xra{D_{k}}\mathcal{D}'$, for $k=1,\ldots,n$,
are the partial derivatives of distributions. Once again, each $D_{k}$
is a sheaf morphism because $D_{kU}:\mathcal{D}'(U)\ra\mathcal{D}'(U)$
for all open sets $U\subseteq\R^{n}$, and they commute with restrictions
of distributions: $D_{kV}(\mathcal{D}'_{UV}(T))=D_{k}(T|_{V})=\mathcal{D}'_{UV}(D_{kU}(T))=D_{k}(T)|_{V}$
if $V\subseteq\R^{n}$ is open and $V\subseteq U$. Moreover, $D_{kU}$
is compatible with partial derivatives of $\mathcal{C}_{k}^{1}$ functions,
i.e.~$\lambda(\partial_{k}f)=D_{k}(\lambda(f))$ or, specifying all
the domains and inclusions $\lambda_{U}\left(\partial_{kU}f\right)=D_{kU}\left(\lambda_{U}\left(\iota_{kU}(f)\right)\right)$.
In the following, \textcolor{black}{we use the notations $D_{kU}^{j}:=D_{kU}\circ\ptind^{j}\circ D_{kU}$
and $D_{U}^{\alpha}:=D_{1U}^{\alpha_{1}}\circ\ldots\circ D_{nU}^{\alpha_{n}}$
for any multi-index $\alpha\in\N^{n}$ and any open set }$U\subseteq\R^{n}$.
Note explicitly that $D_{U}^{\alpha}(\lambda_{U}(f))=\lambda_{U}(\partial^{\alpha}f)$
if $f\in\mathcal{C}^{\alpha}(U)$.
\item \label{enu:poly}If $\alpha\in\N^{n}$, $f$, $g\in\mathcal{C}^{0}(U)$
and $U$ is an $n$-dimensional interval, then $D_{U}^{\alpha}(\lambda_{U}(f))=D_{U}^{\alpha}(\lambda_{U}(g))$
holds if and only if we can write $f-g=\theta_{1}+\ldots+\theta_{n}$,
where each $\theta_{k}$ is a polynomial in $x_{k}$ of degree $<\alpha_{k}$
whose coefficients are continuous functions on $U$ independent by
$x_{k}$.
\item $D_{h}\circ D_{k}=D_{k}\circ D_{h}$ for all $h$, $k=1,\ldots,n$.
\end{enumerate}
\end{rem}

\begin{thm}
\textcolor{black}{\label{thm:UP of D'}$(\mathcal{D}',\lambda,(D_{k})_{k})$
is a co-universal solution of the problem $\mathcal{P}(H,j,(\delta_{k})_{k})$
given by:}
\begin{enumerate}
\item \textcolor{black}{\label{enu:distrSheaf}$H:\left(\R^{n}\right)^{\text{op}}\ra\mathbf{Vect}_{\R}$
is a sheaf of real vector spaces.}
\item \textcolor{black}{\label{enu:distrMorph}$j:\mathcal{C}^{0}\longrightarrow H$
is a sheaf morphism.}
\item \textcolor{black}{\label{enu:distrComp}$\delta_{k}:H\longrightarrow H$,
$k=1,\ldots,n$, are compatible with partial derivatives of $\mathcal{C}_{k}^{1}$
functions: $\delta_{k}\circ j\circ\iota_{k}=j\circ\partial_{k}$,
i.e.~the following diagram of sheaves morphisms commutes for all
$k=1,\ldots,n$: 
\[
\xymatrix{\mathcal{C}_{k}^{1}\ar@{^{(}->}[r]^{\iota_{k}}\ar[dr]_{\partial_{k}} & \mathcal{C}^{0}\ar[r]^{j} & H\ar[d]^{\delta_{k}}\\
 & \mathcal{C}^{0}\ar[r]^{j} & H
}
\]
}
\item \label{enu:distPoly}Let $\alpha\in\N^{n}$, $f\in\mathcal{C}^{0}(U)$
and $U$ be an \textcolor{black}{$n$-dimensional interval. Assume
that $f=\theta_{1}+\ldots+\theta_{n}$, where each $\theta_{k}$ is
a polynomial in $x_{k}$ of degree $<\alpha_{k}$ whose coefficients
are continuous functions on $U$ independent by $x_{k}$, then $\delta_{U}^{\alpha}(j_{U}(f))=0$.}
\item \label{enu:distrComm}$\delta_{h}\circ\delta_{k}=\delta_{k}\circ\delta_{h}$
for all $h$, $k=1,\ldots,n$.
\end{enumerate}
\noindent \textcolor{black}{The problem is solvable with respect to
the property $\mathcal{Q}(\psi,H,j,(\delta_{k})_{k},\overline{H},\overline{\jmath},(\overline{\delta}_{k})_{k})$
to preserve embeddings and derivatives given by 
\[
\psi:H\longrightarrow\overline{H},\qquad\psi\circ j=\overline{\jmath,}\qquad\psi\circ\delta_{k}=\overline{\delta}_{k}\circ\psi\quad\forall k=1,\ldots,n,
\]
i.e.~when the following diagrams of sheaves morphisms commute}

\textcolor{black}{
\[
\xymatrix{H\ar[d]_{\psi}\ar[r]^{\delta_{k}} & H\ar[d]^{\psi}\\
\overline{H}\ar[r]_{\overline{\delta}_{k}} & \overline{H}
}
\hfill\hfill\xymatrix{\mathcal{C}^{0}\ar[r]^{j}\ar[rd]^{\overline{\jmath}} & H\ar[d]^{\psi}\\
 & \overline{H}
}
\]
Therefore, if $(H,j,(\delta_{k})_{k})$ is any solution }of \ref{enu:distrSheaf}-\ref{enu:distrComp},
th\textcolor{black}{en}

\textcolor{black}{
\begin{equation}
\exists!\psi:\mathcal{D}'\longrightarrow H:\ j=\psi\circ\lambda,\ \psi\circ D_{k}=\delta_{k}\circ\psi\quad\forall k=1,\ldots,n.\label{eq:distrUnique}
\end{equation}
}

\end{thm}

\begin{proof}
\textcolor{black}{We only have to prove \eqref{eq:distrUnique}, because
it is clear from \cite{Sch50} that $(\mathcal{D}',\lambda,(D_{k})_{k})$
is a solution of }\ref{enu:distrSheaf}-\ref{enu:distrComp}\textcolor{black}{.}

\textcolor{black}{Let $U\subseteq\R^{n}$ and let $T\in\mathcal{D}'(U)$.
The key idea to define $\psi_{U}(T)$ is to use the local structure
of distributions to define $\psi_{C}(T|_{C})$ for any $C\subseteq U$
with $\overline{C}\Subset U$, and then to use the gluing property
to define $\psi_{U}(T)$ as the gluing of the compatible family $(\psi_{C}(T|_{C}))_{C}$.}

\textcolor{black}{The local structure theorem of distributions (see
\cite{Sch50}) yields 
\begin{equation}
T|_{C}=D_{C}^{\alpha}(\lambda_{C}(f))\label{eq:local structure}
\end{equation}
for some multi-index $\alpha\in\N^{n}$ and some continuous function
$f\in\mathcal{C}^{0}(C)$. We necessarily have to define
\begin{equation}
\psi_{C}(T|_{C}):=\delta_{C}^{\alpha}\left(j_{C}(f)\right)\in H(C),\label{eq:psi on finite order}
\end{equation}
but we clearly have to prove that this definition does not }depend
on $\alpha$ and $f$ in \eqref{eq:local structure}. Assume that
\begin{equation}
T|_{C}=D_{C}^{\alpha'}\left(\lambda_{C}(g)\right),\label{eq:localStructure2}
\end{equation}
for another $\alpha'\in\mathbb{N}^{n}$ and another $g\in\mathcal{C}^{0}(C)$.
We claim that \textcolor{black}{
\begin{equation}
\delta_{C}^{\alpha}\left(j_{C}(f)\right)=\delta_{C}^{\alpha'}\left(j_{C}(g)\right).\label{eq: independence from rep}
\end{equation}
Indeed, since $\delta^{\alpha}\circ j$ is a sheaf of morphisms, and
since the set of all the n-dimensional intervals included in $C$
is a covering of $C$, it is sufficient to show that }\eqref{eq: independence from rep}\textcolor{black}{{}
holds for any $n$-dimensional interval $C=(c_{1}-r,c_{1}+r)\times\ptind^{n}\times(c_{n}-r,c_{n}+r)$
of center $c\in\R^{n}$ and sides $2r\in\R_{>0}$. We first prove
that we can change the functions $f$ and $g$ so that to have the
same multi-index $\alpha=\alpha'$. Assume e.g.~that $\alpha'_{k}>\alpha_{k}$,
set $a_{k}:=\alpha'_{k}-\alpha_{k}$, and integrate $f$ in the variable
$x_{k}$ for $a_{k}$ times:}
\begin{equation}
\bar{f}(x):=\int_{c_{k}}^{x_{k}}...\int_{c_{k}}^{t_{2}}\int_{c_{k}}^{t_{1}}f(x_{1},\ldots,x_{k-1},t_{0},x_{k+1},\ldots,x_{n})\,\diff t_{0}\,\diff t_{1}...\,\diff t_{a_{k}-1}\quad\forall x\in C.\label{eq:primitive multivariable func}
\end{equation}
The function $\overline{f}$ is well-defined because $C$ is an $n$-dimensional
interval of c\textcolor{black}{enter $c$ and we have $\bar{f}\in\mathcal{C}_{k}^{a_{k}}(C)$
and $\partial_{k}^{a_{k}}\bar{f}=f$. Therefore, using the compatibility
of $D_{k}$ and $\partial_{k}$, we get
\begin{equation}
D_{C}^{\alpha}(\lambda_{C}(f))=D_{C}^{\alpha}(\lambda_{C}(\partial_{k}^{a_{k}}\bar{f}))=D_{C}^{\alpha}(D_{kC}^{a_{k}}(\lambda_{C}\bar{f}))=D_{C}^{\alpha+a_{k}e_{k}}(\lambda_{C}(\bar{f})),\label{eq:incr_a_k}
\end{equation}
and 
\[
\alpha+a_{k}e_{k}=(\alpha_{1},\ldots,\alpha_{k-1},\alpha'_{k},\alpha_{k+1},\ldots,\alpha_{k}).
\]
If $\alpha'_{k}<\alpha_{k}$, we can proceed similarly using \eqref{eq:localStructure2}
instead of \eqref{eq:local structure}. Therefore, for $\bar{\alpha}_{k}:=\max(\alpha_{k},\alpha'_{k})$,
$\alpha_{k}^{f}:=\max(\alpha'_{k}-\alpha_{k},0)$, $\alpha_{k}^{g}:=\max(\alpha{}_{k}-\alpha'_{k},0)$
and for suitable $\bar{f}\in\mathcal{C}^{\alpha^{f}}(C)$, $\bar{g}\in\mathcal{C}^{\alpha^{g}}(C)$,
we have
\[
T|_{C}=D_{C}^{\bar{\alpha}}\left(\lambda_{C}(\bar{f})\right)=D_{C}^{\bar{\alpha}}\left(\lambda_{C}(\bar{g})\right),
\]
i.e.~$D_{C}^{\bar{\alpha}}\left(\lambda_{C}(\bar{f}-\bar{g})\right)=0$.
Therefore, Rem.~\ref{rem:Schwartz's-solution}.\ref{enu:poly} (the
necessary condition part) yields that $\bar{f}-\bar{g}$ can be written
(on $C$) as $\bar{f}-\bar{g}=\theta_{1}+\ldots+\theta_{n}$, where
each $\theta_{k}$ is a polynomial in $x_{k}$ of degree $<\bar{\alpha}_{k}$
whose coefficients are continuous functions on $C$ independent by
$x_{k}$. Property \ref{enu:distPoly} for $(H,j,(\delta_{k})_{k})$
(note explicitly that this condition only states the sufficient part
of Rem.~\ref{rem:Schwartz's-solution}.\ref{enu:poly}) implies $\delta_{C}^{\bar{\alpha}}(j_{C}(\bar{f}-\bar{g}))=0$,
and hence $\delta_{C}^{\bar{\alpha}}(j_{C}(\bar{f}))=\delta_{C}^{\bar{\alpha}}(j_{C}(\bar{g}))$
because we are considering sheaves of vector spaces. Exactly as we
increased $\alpha_{k}$ by $a_{k}$ (if $\alpha'_{k}>\alpha_{k}$)
in \eqref{eq:incr_a_k}, we can now proceed backward to return to
the old multi-index: since $\bar{f}\in\mathcal{C}_{k}^{a_{k}}(C)$
\[
\delta_{C}^{\bar{\alpha}}(j_{C}(\bar{f}))=\delta_{C}^{\bar{\alpha}-a_{k}e_{k}}(\delta_{kC}^{a_{k}}(j_{C}\bar{f}))=\delta_{C}^{\bar{\alpha}-a_{k}e_{k}}(j_{C}(\partial_{k}^{a_{k}}\bar{f})),
\]
and by induction we get $\delta_{C}^{\bar{\alpha}}(j_{C}(\bar{f}))=\delta_{C}^{\alpha}(j_{C}(\partial^{\overline{\alpha}-\alpha}\overline{f}))=\delta_{C}^{\alpha}(j_{C}(f))$.
This proves that $\delta_{C}^{\alpha}\left(j_{C}(f)\right)=\delta_{C}^{\alpha'}\left(j_{C}(g)\right)$,
and hence our claim is proved.}

We denote by $B(U)$ the set of all the relatively compact sets of
$U$, which is, by the local structure of distributions, a covering
of $U$. The family $\left(\psi_{C}(T|_{C})\right)_{C\in B(U)}$ is
a compatible one. In fact $H_{C',C'\cap C}\left(\psi_{C'}(T|_{C'})\right)=H_{C',C'\cap C}\left(\delta_{C'}^{\alpha}(j_{C'}(f))\right)=\delta_{C'\cap C}^{\alpha}\left(j_{C'\cap C}(f)\right)=H_{C,C\cap C'}\left(\psi_{C}(T|_{C})\right)$,
and we can hence set
\[
\psi_{U}(T):=H_{U}\left[\left(\psi_{C}(T|_{C})\right)_{C\in B(U)}\right]\quad\forall T\in\mathcal{D}'(U).
\]
We claim that if $T:=D_{U}^{\alpha}(\lambda_{U}(f))$ for some $\alpha\in\N^{n}$
and for some $f\in\mathcal{C}^{0}(U)$, then $\psi_{U}(T)=\delta_{U}^{\alpha}(j_{U}(f))$.
Indeed, for any $V\in B(U)$ we have 
\begin{multline*}
H_{U,V}(\psi_{U}(T))=H_{U,V}\left(H_{U}\left[\left(\delta_{C}^{\alpha}(j_{C}(f))\right)_{C\in B(U)}\right]\right)\\
=H_{V}\left[\left(\delta_{C\cap V}^{\alpha}(j_{C\cap V}(f))\right)_{C\in B(U))}\right]=\delta_{V}^{\alpha}(j_{V}(f))=H_{U,V}(\delta_{U}^{\alpha}(j_{U}(f))).
\end{multline*}
where we used \eqref{eq:glueFam} in the second equality. Thus, by
the locally condition of $H$ (see Def.~\ref{def:sheaf}.\ref{enu:sheaf }.\eqref{enu:sheaf 1}),
our claim is proved. It follows in particular that $\psi_{U}(\lambda_{U}(f))=\delta_{U}^{0}(j_{U}(f))=j_{U}(f)$
for all $f\in\mathcal{C}^{0}(U)$.

If $C$, $C'\in B(U)$ are such that $C'\subseteq C$ then for any
$T\in\mathcal{D}'(U)$
\begin{align}
H_{C,C'}\left(\psi_{C}(T|_{C})\right) & =H_{C,C'}\left(\delta_{C}^{\alpha}(j_{C}(f))\right)=\delta_{C'}^{\alpha}\left(j_{C'}(f)\right)=\psi_{C'}\left(D_{C'}^{\alpha}\left(\lambda_{C'}(f)\right)\right)\label{eq:sheaf 11}\\
 & =\psi_{C'}(T|_{C'})
\end{align}
because $\delta^{\alpha}\circ j$ is sheaf morphism. Thus, for any
$V\subseteq U$, \eqref{eq:sheaf 11} together with \eqref{eq:glueFam}
imply that 
\begin{align*}
H_{U,V}\left(\psi_{U}(T)\right) & =H_{U,V}\left(H_{U}\left[\left(\psi_{C}(T|_{C})\right)_{C\in B(U)}\right]\right)\\
 & =H_{V}\left[\left(H_{C,V\cap C}(\psi_{C}(T|_{C}))\right)_{C\in B(U)}\right]\\
 & =H_{V}\left[\left(\psi_{V\cap C}(T|_{V\cap C})\right)_{C\in B(U)}\right]=H_{V}\left[\left(\psi_{D}(T|_{D})\right)_{D\in B(V)}\right].
\end{align*}
where the latter equality follows from the fact that $H$ is sheaf
morphism, and the fact that the families $\left(\psi_{C\cap V}(T|_{C\cap V})\right)_{C\in B(U)}$,
$\left(\psi_{D}(T|_{D})\right)_{D\in B(V)}$ are compatible and locally
equal. Therefore $\psi:\mathcal{D}'\ra H$ is a sheaf morphism. To
prove the equality $\psi\circ D_{k}=\delta_{k}\circ\psi$, we have
\begin{align*}
\psi_{U}(D_{kU}(T)) & =H_{U}\left[\left(\psi_{C}(D_{kU}T)|_{C}\right)_{C\in B(U)}\right]=H_{U}\left[\left(\psi_{C}(D_{kC}(T|_{C}))\right)_{C\in B(U)}\right]\\
 & =H_{U}\left[\left(\delta_{kC}(\psi_{C}(T|_{C})\right)_{C\in B(U)}\right]=\delta_{kU}(\psi_{U}(T)),
\end{align*}
where we used the equality 
\[
\psi_{C}(D_{kC}(T|_{C})=\psi_{C}(D_{C}^{e_{k}+\alpha}(\lambda_{C}(f)))=\delta_{kC}\delta_{C}^{\alpha}(j_{C}(f))=\delta_{kC}\psi_{C}(T|_{C})
\]
for some continuous function $f\in\mathcal{C}(U)$ and multi-index
$\alpha\in\N^{n}$. \textcolor{black}{Note explicitly that in the
step $\delta_{C}^{\alpha+e_{k}}=\delta_{kC}\circ\delta_{C}^{\alpha}$
above we need the commutativity property \ref{enu:distrComm}.}

It remains to prove the uniqueness. Assume that also $\bar{\psi}$
satisfies \eqref{eq:distrUnique}; let $C\in B(U)$ and let $f$ and
$\alpha$ be such that $T|_{C}=D_{C}^{\alpha}(\lambda_{C}(f))$, then
\[
\bar{\psi}_{C}(T|_{C})=\bar{\psi}_{C}(D_{C}^{\alpha}(\lambda_{C}(f)))=\delta_{C}^{\alpha}(\bar{\psi}_{C}(\lambda_{C}(f)))=\delta_{C}^{\alpha}(j_{C}(f))=\psi_{C}(T|_{C}).
\]
Therefore, property \eqref{eq:morphGlue} yields
\begin{multline*}
\bar{\psi}_{U}(T)=\bar{\psi}_{U}\left(\mathcal{D}'_{U}\left[\left(T|_{C}\right)_{C\in B(U)}\right]\right)=H_{U}\left[\left(\bar{\psi}_{C}(T|_{C})\right)_{C\in B(U)}\right]=\\
H_{U}\left[\left(\psi_{C}(T|_{C})\right)_{C\in B(U)}\right]=\psi_{U}(T).
\end{multline*}
\end{proof}
Using a categorical language, the universal property Thm.~\ref{thm:UP of D'}
(and the general Thm.~\ref{thm:UPuniqueUpToIso}) corresponds to
the axiomatic characterization of distributions given by Sebastiao
e Silva in \cite{SeS61,SeS64}. However, note that Thm.~\ref{thm:UP of D'}
yields a characterization up to isomorphisms of the whole sheaf of
distributions, not only those defined locally as in \cite{SeS61,SeS64}.
Moreover, note that the universal property allows one to avoid both
the axiom of local structure of distributions \cite[Axiom~3]{SeS64},
and the necessary condition of Rem.~\ref{rem:Schwartz's-solution}.\ref{enu:poly}
(see \cite[Axiom~4]{SeS64}). In fact, we have the following
\begin{cor}
\label{cor:locStr-Poly}If $(H,j,(\delta_{k})_{k})$ is a co-universal
solution of the problem stated in Thm.~\ref{thm:UP of D'}, then
\begin{enumerate}
\item \label{enu:locStr}If $U\subseteq\R^{n}$ and $C$ is a relatively
compact set of $U$, then $\forall w\in H(U)\,\exists\alpha\in\N^{n}\,\exists f\in\mathcal{C}^{0}(C):\ w|_{C}=\delta_{C}^{\alpha}(j_{C}(f))$.
\item \label{enu:polyNec}If $\alpha\in\N^{n}$, $f$, $g\in\mathcal{C}^{0}(U)$,
$U$ is an $n$-d\textcolor{black}{imensional interval, and $\delta_{U}^{\alpha}(j_{U}(f))=\delta_{U}^{\alpha}(j_{U}(g))$
then we can write $f-g=\theta_{1}+\ldots+\theta_{n}$, where each
$\theta_{k}$ is a polynomial in $x_{k}$ of degree $<\alpha_{k}$
whose coefficients are continuous functions on $U$ independent by
$x_{k}$.}
\end{enumerate}
\end{cor}

\begin{proof}
In fact, Thm.~\ref{thm:UPuniqueUpToIso} yields an isomorphism $\psi:\mathcal{D}'\longrightarrow H$
which preserves derivatives $\psi\circ D_{k}=\delta_{k}\circ\psi$
and embeddings $j=\psi\circ\lambda$. Therefore, the equalities $\psi^{-1}(w|_{C})=D_{C}^{\alpha}(\lambda_{C}(f))$
and $\delta_{U}^{\alpha}(j_{U}(f))=\delta_{U}^{\alpha}(j_{U}(g))$
are equivalent to $w|_{C}=\delta_{C}^{\alpha}(j_{C}(f))$ and $D_{U}^{\alpha}(\lambda_{U}(f))=D_{U}^{\alpha}(\lambda_{U}(g))$.
The claims then follow from similar properties of $(\mathcal{D}',\lambda,(D_{k})_{k})$.
\end{proof}

\subsection{\label{subsec:Application}Application: Sebastiao e Silva algebraic
definition of distributions}

In the study of universal properties, it frequently happens that this
characterization (up to isomorphisms) suggests possible generalizations.
For distributions, these ideas are actually already contained in the
proof of Thm.~\ref{thm:UP of D'}, but we prefer to explain them
using the thoughts of Sebastiao e Silva, see \cite{SeS61,SeS64}.
Assume that the open set $I$ is an $n$-dimensional interval $I=(c_{1}-r,c_{1}+r)\times\ptind^{n}\times(c_{n}-r,c_{n}+r)$.
For each continuous function $f\in\mathcal{C}^{0}(I)$ and each $k=1,\ldots,n$,
we can consider any primitive of $f$ with respect to the variable
$x_{k}$, e.g.~setting
\begin{equation}
\mathfrak{I}_{k}f(x):=\int_{c_{k}}^{x_{k}}f(x_{1},\ldots,x_{k-1},t,x_{k+1},\ldots,x_{n})\,\diff t,\label{eq:I_k}
\end{equation}
and, more generally, $\mathfrak{I}^{\alpha}:=\mathfrak{I}_{1}^{\alpha_{1}}\circ\ldots\circ\mathfrak{I}_{n}^{\alpha_{n}}$
for $\alpha\in\N^{n}$, so that $\partial^{\beta}\mathfrak{J}^{\alpha}f=\mathfrak{I}^{\alpha-\beta}f$
if $\alpha\ge\beta$ and $f\in\mathcal{C}^{0}(I)$. Assume that $f$,
$g\in\mathcal{C}^{0}(I)$, $r$, $s\in\N^{n}$, and $D^{r}f=D^{s}g$
(in the sense of distributions; for simplicity, here we omit the dependence
on the open set $I$ and we identify $\lambda_{I}(f)$ with $f$).
Set $m:=\max(r,s)$, then the compatibility property Thm.~\ref{thm:UP of D'}.\ref{enu:distrComp}
yields $D^{r}f=D^{r}(\partial^{m-r}(\mathfrak{I}^{m-r}f))=D^{m}(\mathfrak{I}^{m-r}f)=D^{m}(\mathfrak{I}^{m-s}g)=D^{s}g$.
Therefore, Cor.~\ref{cor:locStr-Poly}.\ref{enu:polyNec} yields
\textcolor{black}{$\mathfrak{I}^{m-r}f-\mathfrak{I}^{m-s}g=\theta_{1}+\ldots+\theta_{n}$,
where each $\theta_{k}$ is a polynomial in $x_{k}$ of degree $<m_{k}$
whose coefficients are continuous functions on $I$ independent by
$x_{k}$. Denote by $\mathcal{P}_{m}$ the set of all the functions
$\theta$ of this form $\theta=\theta_{1}+\ldots+\theta_{n}$. Therefore,
we proved that}
\begin{equation}
D^{r}f=D^{s}g\ \iff\ \mathfrak{I}^{m-r}f-\mathfrak{I}^{m-s}g\in\mathcal{P}_{m},\ \text{where }m:=\max(r,s).\label{eq:SebSilIdea}
\end{equation}
The main idea of \cite{SeS61,SeS64} is that a condition such as the
right hand side of \eqref{eq:SebSilIdea} can be stated for pair of
continuous functions without any need to use methods of functional
analysis, but only using a \emph{formal} algebraic approach: we can
say that the derivative $D^{r}f$ of a continuous function $f\in\mathcal{C}^{0}(I)$
is simply a formal operation corresponding to the pair $(r,f)$, and
two pairs are equivalent if the right hand side of \eqref{eq:SebSilIdea}
holds. Therefore, if $I$ is an $n$-dimensional interval, we can
define: $(r,f)\sim(s,g)$ if $r$, $s\in\N^{n}$, $f$, $g\in\mathcal{C}^{0}(I)$
and $\mathfrak{I}^{m-r}f-\mathfrak{I}^{m-s}g\in\mathcal{P}_{m}$,
where $m:=\max(r,s)$; $\mathcal{D}_{\text{f}}'(I):=(\N^{n}\times\mathcal{C}^{0}(I))/\sim$;
$\lambda_{I}(f):=[(0,f)]_{\sim}$; $D_{k}([(r,f)]_{\sim}):=\left[(r+e_{k},f)\right]_{\sim}$,
so that $D^{r}f=[(r,f)]_{\sim}\in\mathcal{D}_{\text{f}}'(I)$; finally,
the vector space operations are defined as $D^{r}f+D^{s}g:=D^{m}(\mathfrak{I}^{m-r}f+\mathfrak{I}^{m-s}g)$
($m:=\max(r,s)$) and $\mu\cdot D^{r}f:=D^{r}(\mu f)$ for all $\mu\in\R$;
the restriction to another $n$-dimensional interval $J\subseteq I$
is defined by $\left(D^{r}f\right)|_{J}:=D^{r}(f|_{J})$. With these
definitions we obtain a functor
\begin{equation}
\mathcal{D}'_{\text{f}}:\mathcal{I}(\R^{n})^{\text{op}}\ra\mathbf{Vect}_{\R},\label{eq:functPoset}
\end{equation}
where $\mathcal{I}(U)$ is the poset of all the $n$-dimensional intervals
contained in the open set $U\subseteq\R^{n}$. Clearly, $\mathcal{I}(\R^{n})$
is not a topological space, but it is a base for the Euclidean topology
of $\R^{n}$, and this suffices to apply a general co-universal method
(called \emph{sheafification}, see \cite{Bor94,KaSc}) to associate
a sheaf $\mathcal{D}':(\R^{n})^{\text{op}}\ra\mathbf{Vect}_{\R}$
to $\mathcal{D}'_{\text{f}}$: this corresponds to the intuitive idea
that any distribution is obtained by gluing a compatible family, where
each element of the family is the (distributional) derivative of a
continuous function. We first use distribution theory as an example
to motivate sheafification in this case, but then we introduce this
construction in general terms as another example to solve a problem
in the simplest way.

For an arbitrary $T\in\mathcal{D}'(U)$, $U\subseteq\R^{n}$ being
an open set, we can consider all the possible intervals $I\in\mathcal{I}(U)$
such that $T|_{I}$ is in $\mathcal{D}'_{\text{f}}(I)$:
\begin{equation}
\mathcal{B}(T):=\left\{ I\in\mathcal{I}(U)\mid T|_{I}\in\mathcal{D}'_{\text{f}}(I)\right\} .\label{eq:BT}
\end{equation}
By the local structure of distributions, and the fact that $\mathcal{I}(U)$
is a base, we have that $\mathcal{B}(T)$ is a covering of $U$. Intuitively,
among all the possible coverings of $U$ made of intervals, $\mathcal{B}(T)$
is the largest one (e.g.~it surely contains all the $I\in\mathcal{I}(U)$
such that $\bar{I}\subseteq U$ where the local structure theorem
applies). We start by understanding how to formalize this idea that
$\mathcal{B}(T)$ is ``the largest one'' because this would allow
us to use only the separateness of $\mathcal{D}'_{\text{f}}(-)$ and
an arbitrary $\mathcal{B}(T)$-indexed compatible family such as $\left(T|_{I}\right)_{I\in\mathcal{B}(T)}$.
\begin{rem}
\label{rem:sepComFctr}Separateness and being a compatible family
can clearly be formulated also for a functor of the type \eqref{eq:functPoset}:
\begin{enumerate}
\item We say that $\mathcal{D}'_{\text{f}}$ is \emph{separated} because
if $T$, $S\in\mathcal{D}'_{\text{f}}(I)$, $I\in\mathcal{I}(\R^{n})$,
and $(I_{j})_{j\in J}$ is a covering of $I$ made of intervals such
that $T|_{I_{j}}=S|_{I_{j}}$ for all $j\in J$, then $T=S$.
\item For all $I\in\mathcal{B}(T)$, we have $T|_{I}\in\mathcal{D}'_{\text{f}}(I)$;
moreover, $\left(T|_{I}\right)|_{K}=\left(T|_{J}\right)|_{K}$ for
all $I$, $J\in\mathcal{B}(T)$ and all $K\in\mathcal{I}(\R^{n})$
such that $K\subseteq I\cap J$, i.e.~$\left(T|_{I}\right)_{I\in\mathcal{B}(T)}$
is a \emph{compatible family}.
\end{enumerate}
\end{rem}

Now, let $S\in\mathcal{D}'_{\text{f}}(J)$, $J\in\mathcal{I}(U)$;
assume that $S$ is \emph{locally equal to} $\left(T|_{I}\right)_{I\in\mathcal{B}(T)}$,
i.e.~it satisfies
\begin{equation}
\forall I\in\mathcal{B}(T)\,\forall K\in\mathcal{I}(\R^{n}):\ K\subseteq I\cap J\ \Rightarrow\ S|_{K}=T|_{K},\label{eq:locEq}
\end{equation}
then by the sheaf property of $\mathcal{D}'$, we have $S=T|_{J}$
and hence $J\in\mathcal{B}(T)$: in these general sheaf-theoretical
terms the covering $\mathcal{B}(T)$ is the largest one. It clearly
also holds the opposite implication: if $J\in\mathcal{B}(T)$, then
$S:=T|_{J}$ satisfy \eqref{eq:locEq}. We write $S=_{J}\left(T|_{I}\right)_{I\in\mathcal{B}(T)}$
if \eqref{eq:locEq} holds, so that
\[
\mathcal{B}(T)=\left\{ J\in\mathcal{I}(U)\mid\exists S\in\mathcal{D}'_{\text{f}}(J):\ S=_{J}\left(T|_{I}\right)_{I\in\mathcal{B}(T)}\right\} .
\]
Intuitively, we can say that the distribution $T$ can be identified
with the family $\left(T|_{I}\right)_{I\in\mathcal{B}(T)}$ defined
on the largest possible domain (in this sense, we expect it is co-universal).

All this motivates the following general
\begin{defn}
\label{def:locEQMaxFam}Let $\mathcal{I}$ be a base for the topological
space $\mathbb{T}$, $P:\mathcal{I}^{\text{op}}\ra\mathbf{Vect}_{\R}$
a functor, $U\in\mathbb{T}$, $\mathcal{B}\subseteq\mathcal{I}$ be
a covering of $U$, $J\in\mathcal{I}$, $S\in P(J)$, and $(T_{I})_{I\in\mathcal{B}}$
a $P$-compatible family. Then, we write $S=_{J}(T_{I})_{I\in\mathcal{B}}$
and we say $S$ \emph{locally equals }$(T_{I})_{I\in\mathcal{B}}$
\emph{on }$J$ if and only if 
\[
\forall I\in\mathcal{B}\,\forall K\in\mathcal{I}:\ K\subseteq I\cap J\ \Rightarrow\ P_{J,K}(S)=P_{I,K}(T_{I}).
\]
Moreover, we say that $(T_{I})_{I\in\mathcal{B}}$ \emph{is a maximal
family on} $U$ if and only if
\begin{enumerate}
\item $(T_{I})_{I\in\mathcal{B}}$ is a compatible family
\item $\forall J\in\mathcal{I}\,\forall S\in P(J):\ S=_{J}(T_{I})_{I\in\mathcal{B}}\ \Rightarrow\ J\in\mathcal{B},\ S=T_{J}$.
\end{enumerate}
\end{defn}

The separateness of $P$ is used in the following result, that allows
us to consider the maximal family generated by a given compatible
family. The idea is to consider all the section $S\in P(J)$ of the
presheaf that locally equals the given family.
\begin{thm}
\label{thm:sheafification1}Let $\mathcal{I}$ be a base for the topological
space $\mathbb{T}$, $P:\mathcal{I}^{\text{op}}\ra\mathbf{Vect}_{\R}$
a separated functor, $\mathcal{B}\subseteq\mathcal{I}$ be a covering
of $U\in\mathbb{T}$ and let $(T_{I})_{I\in\mathcal{B}}$ be a compatible
family. Set 
\begin{equation}
\bar{\mathcal{B}}:=\left\{ J\in\mathcal{I}\mid\exists S\in P(J):\ S=_{J}(T_{I})_{I\in\mathcal{B}}\right\} \label{eq:barB}
\end{equation}
then, we have
\begin{enumerate}
\item \label{enu:barT}$\forall J\in\bar{\mathcal{B}}\,\exists!\,\bar{T}\in P(J):\ \bar{T}=_{J}(T_{I})_{I\in\mathcal{B}}.$
We denote by $\bar{T}_{J}$ this unique $\overline{T}$.
\item \label{enu:barT-ext}$\forall I\in\mathcal{B}:\ I\in\bar{\mathcal{B}}$
and $\bar{T}_{I}=T_{I}$.
\item \label{enu:barT-max}$(\bar{T}_{I})_{I\in\bar{\mathcal{B}}}$ is a
maximal family on $U$.
\end{enumerate}
\end{thm}

\begin{proof}
To prove \ref{enu:barT} simply use \eqref{eq:barB} and the separateness
of $P$. To prove \ref{enu:barT-ext} use the assumption that $(T_{I})_{I\in\mathcal{B}}$
is a compatible family. To prove \ref{enu:barT-max} use \eqref{eq:barB}
and Def.~\ref{def:locEQMaxFam}.
\end{proof}
\noindent We use the notation $\text{max}[(T_{I})_{I\in\mathcal{B}}]:=(\bar{T}_{I})_{I\in\bar{\mathcal{B}}}$,
and we can now define the sheaf $\overline{P}$ on objects:
\begin{defn}
If $U\in\mathbb{T}$, set $(T_{I})_{I\in\mathcal{B}}\in\overline{P}(U)$
if and only if
\begin{enumerate}
\item $\mathcal{B}\subseteq\mathcal{I}$ is a covering of $U$;
\item $(T_{I})_{I\in\mathcal{B}}$ is a maximal family on $U$.
\end{enumerate}
\end{defn}

\noindent To eventually get an $R$ module (which is the case of real-valued
distributions), we also have to define module operations:
\begin{defn}
Let $U\in\mathbb{T}$, $r\in R$ and let $(T_{I})_{I\in\mathcal{B}}$,
$(S_{J})_{J\in\mathcal{C}}\in\overline{P}(U)$. Then
\begin{enumerate}
\item $(T_{I})_{I\in\mathcal{B}}+(S_{J})_{J\in\mathcal{C}}:=\text{max}[(T_{A}+S_{A})_{A\in\mathcal{B}\cap\mathcal{C}}]$,
where $\mathcal{B}\cap\mathcal{C}:=\{I\cap J\mid I\subseteq\mathcal{B},\,J\subseteq\mathcal{C}\}$
which is clearly a covering of $U$; clearly the family $(T_{A}+S_{A})_{A\in\mathcal{B}\cap\mathcal{C}}$
is a compatible one.
\item $r\cdot(T_{I})_{I\in\mathcal{B}}:=\text{max}[(r\cdot T_{I})_{I\in\mathcal{B}}]$.
\end{enumerate}
\end{defn}

\noindent Using these operations, it is possible to prove that $(\overline{P}(U),+,\cdot)\in\mathbf{Mod}_{R}$.
We still use the symbol $\overline{P}(U)$ to denote this $R$-module.
We finally define $\overline{P}$ on arrows.
\begin{defn}
Let $U$, $V\in\mathbb{T}$, $V\subseteq U$. Then
\begin{enumerate}
\item $\mathcal{C}_{\subseteq V}:=\{J\subseteq V\mid J\in\mathcal{C}\}$
where $\mathcal{C}\subseteq\mathcal{I}$ is a covering of $U$.
\item $\overline{P}_{UV}:(T_{I})_{I\in\mathcal{C}}\in\overline{P}(U)\longmapsto(P_{IJ}(T_{I}))_{J\in\mathcal{C}_{\subseteq V}}\in\overline{P}(V)$,
where $I\in\mathcal{C}$ is any open set such that $I\supseteq J$
(two different of these $I$ yield the same value of $P_{IJ}(T_{I})$
by the compatibility property of $(T_{I})_{I\in\mathcal{C}}\in\bar{P}(U)$).
It is not hard to prove that the family $(P_{IJ}(T_{I}))_{J\in\mathcal{C}_{\subseteq V}}$
is already a maximal one.
\end{enumerate}
\end{defn}

\noindent The link between $P$ and $\overline{P}$ is given by the
following natural transformation 
\begin{equation}
\eta_{I}:T\in P(I)\mapsto\text{max}\left[(P_{IJ}(T))_{J\in\mathcal{I}_{\subseteq I}}\right]\in\overline{P}(I).\label{eq:natural transf.}
\end{equation}
 With these definitions, we have the following universal property,
whose proof easily follows from our definitions and from Thm.~\ref{thm:sheafification1}:
\begin{thm}
\label{thm:sheafification }If $P:\mathcal{I}^{\text{op}}\longrightarrow\textbf{Mod}_{R}$
is separated then
\begin{enumerate}
\item \label{enu:sheafification}$\overline{P}:\mathbb{T}^{\text{op}}\longrightarrow\textbf{Mod}_{R}$
is a sheaf
\item \label{enu:natural transf.}\eqref{eq:natural transf.} defines a
natural transformation
\item $(\overline{P},\eta)$ is co-universal among all $(\overline{P},\eta)$
that satisfy \ref{enu:sheafification}, \ref{enu:natural transf.},
i.e.~if $(\tilde{P},\mu)$ also satisfies \ref{enu:sheafification},
\ref{enu:natural transf.}, then there exists one and only one natural
transformation $\psi$ such that $\psi_{I}\circ\eta_{I}=\mu_{I}$
for all $I\in\mathcal{I}$.
\end{enumerate}
\end{thm}

\noindent The general construction of sheafification of a presheaf
can be found e.g.~in \cite{KaSc,McMo94}. All this formalizes the
intuitive idea that distributions on an arbitrary open set $U$ are
obtained by gluing together \emph{in the simplest way} distributions
on relatively compact $n$-dimensional intervals of $U$.

\subsection{Generalization: distributions on Hilbert spaces}

The previous construction leads naturally to consider several kind
of potential generalizations:
\begin{enumerate}
\item We can think about vector spaces over the complex field $\mathbb{C}$.
Note explicitly that even in $\mathbb{C}^{n}$, the usual construction
of distributions as continuous functionals on compactly supported
smooth functions cannot be generalized to holomorphic maps because
of the identity theorem.
\item The integrals \eqref{eq:I_k} represent a way to construct a primitive
in the direction $e_{k}$ and can hence be generalized to suitable
infinite dimensional spaces.
\item Definition \eqref{eq:I_k} leads us to consider an at most countable
orthonormal family $(e_{k})_{k\in\Lambda}$, $\Lambda\subseteq\N$,
in a Hilbert space, so that orthogonal complement $H=\text{span}(e_{k})\oplus\text{span}(e_{k}){}^{\perp}$
always exists.
\item Multidimensional intervals are used above as a base of the Euclidean
topology, but in more abstract normed spaces balls can be more conveniently
used.
\end{enumerate}
\noindent On the other hand, the definition of a non-trivial space
of generalized functions of a complex variable that allows one to
consider derivatives of continuous functions is a non-obvious task.
In fact, if we want that these generalized functions embed ordinary
continuous maps and, at the same time, satisfy the Cauchy theorem,
then the continuous functions would also be path-independent and,
from Morera's theorem, they would actually be holomorphic functions,
see e.g.~\cite{Zal}. Likewise, if we want that these generalized
functions satisfy the Cauchy-Riemann equations (even with respect
to distributional derivatives), then necessarily the embedded continuous
ones will be ordinary holomorphic functions, see \cite{GrMo}. Using
the language of the present paper, we could say that the co-universal
solution of the problem to have derivatives of continuous functions
of a complex variable which are path-independent or satisfy the Cauchy-Riemann
equation is the sheaf of holomorphic functions, and a larger space
is not possible.

\noindent In the following, we therefore consider a Hilbert space
$H$ with inner product $(x,y)$. The field of scalars is denoted
by $\mathbb{F}\in\{\R,\mathbb{C}\}$. In this space, we fix an orthonormal
Schauder basis $(e_{k})_{k\in\Lambda}$, $\Lambda\subseteq\N$ of
$H$. An interesting example is the space $\mathcal{C}^{0}(K,\R^{d})$
of all the $\R^{d}$ valued continuous functions on a compact set.

For simplicity, we deal with the case $\mathbb{F}=\mathbb{R}$, and
the case $\mathbb{F}=\mathbb{C}$ can be treated in a very similar
way. Let $x$, $c\in H$, $k\in\Lambda$, $J\subseteq\Lambda$ a finite
subset. Under our assumptions, $x=x_{k}+x_{k}^{\perp}$, where $x_{k}=(x,e_{k})e_{k}=:\hat{x}_{k}e_{k}\in\text{span}(e_{k})$
and $x_{k}^{\perp}\in\text{span}(e_{k})^{\perp}$; more generally,
$x=x_{J}+x_{J}^{\perp}$, where $x_{J}:=\sum_{j\in J}x_{j}\in\text{span}(\{e_{j}\}_{j\in J})$,
and $x_{J}^{\perp}\in\text{span}(\{e_{j}\}_{j\in J})^{\perp}$. We
set $\widehat{[c,x]}_{j}:=[\min(\hat{c}_{j},\hat{x}_{j}),\max(\hat{c}_{j},\hat{x}_{j})]\subseteq\R$,
and $[c,x]_{J}=\{\sum_{j\in J}t_{j}e_{j}\mid t_{j}\in\widehat{[c,x]}_{j}\,\,\forall j\in J\}.$

Let $f\in\mathcal{C}^{0}(B_{r}(c),H)$ a continuous function defined
in the ball $B_{r}(c)\subseteq H$ of radius $r>0$ and center $c\in H$.
We have that 
\[
\forall x\in B_{r}(c):\ f(x)=\sum_{k\in\Lambda}\hat{f}_{k}(x)e_{k}.
\]
Using the orthogonality property and the continuity of $f$, one can
see that each $\hat{f}_{k}$ is also continuous. Hence, for any $x\in B_{r}(c)$,
for any $j$, $k\in\Lambda$, the function $\widehat{[c,x]}_{j}\ra\mathbb{R}$,
$t\mapsto\hat{f}_{k}(x_{j}^{\perp}+te_{j})$ is continuous. Therefore,
the integral 
\[
\int_{\hat{c}_{j}}^{\hat{x}_{j}}\hat{f}_{k}(x_{j}^{\perp}+te_{j})\,\diff t
\]
is well defined. We assume that the following assumption holds 
\begin{equation}
\forall x\in B_{r}(c)\,\forall J\subseteq\Lambda\,\text{finite}:\ \sum_{k\in\Lambda}\sup_{y\in[c,x]_{J}}|\hat{f}_{k}(x_{J}^{\perp}+y)|^{2}<\infty.\label{eq: int condition}
\end{equation}
The sheaf of continuous functions $f\in\mathcal{C}^{0}(B_{r}(c),H)$
satisfying \eqref{eq: int condition} is denoted by $\mathcal{C}_{\text{p}}^{0}(B_{r}(c),H)$.
Then, we clearly have 
\[
\left\Vert \sum_{k\in\Lambda}\int_{\hat{c}_{j}}^{\hat{x}_{j}}\hat{f}_{k}(x_{j}^{\perp}+te_{j})\,\diff t\cdot e_{k}\right\Vert ^{2}\leq|\hat{x}_{j}-\hat{c}_{j}|^{2}\sum_{k\in\Lambda}\sup_{y\in[c,x]_{j}}|\hat{f}_{k}(x_{j}^{\perp}+y)|^{2}<\infty.
\]
Therefore, we can set 
\begin{equation}
\mathfrak{I}_{j}(f)(x):=\int_{\hat{c}_{j}}^{\hat{x}_{j}}f(x)\,\diff e_{j}:=\sum_{k\in\Lambda}\int_{\hat{c}_{j}}^{\hat{x}_{j}}\hat{f}_{k}(x_{j}^{\perp}+te_{j})\,\diff t\cdot e_{k}\label{eq:Primitive}
\end{equation}
 and is called the primitive of $f$ in the direction $e_{j}$. Indeed,
\eqref{eq:Primitive} is a generalization of \eqref{eq:I_k}. Moreover,
we have 
\[
\forall x\in B_{r}(c)\,\forall J\subseteq\Lambda\,\text{finite}\,\forall y\in[c,x]_{J}:\ \widehat{\mathfrak{I}_{j}(f)}_{k}(x_{J}^{\perp}+y):=\int_{\hat{c}_{j}}^{\hat{x}_{j}}\hat{f}_{k}((x_{J}^{\perp}+y)_{j}^{\perp}+te_{j})\,\diff t.
\]
We can easily see that 
\[
(x_{J}^{\perp}+y)_{j}^{\perp}:=\left\{ \begin{array}{cc}
x_{J}^{\perp}+y_{j}^{\perp} & j\in J\\
x_{J\cup\{j\}}^{\perp}+y & j\not\in J
\end{array}\right.
\]
In former situation we have 
\begin{multline*}
\sup_{y\in[c,x]_{J}}\left|\widehat{\mathfrak{I}_{j}(f)}_{k}(x_{J}^{\perp}+y)\right|\leq|\hat{x}_{j}-\hat{c}_{j}|\sup_{z\in[c,x]_{J\backslash\{j\}}}\sup_{t\in\widehat{[c,x]}_{j}}\left|\hat{f}_{k}(x_{J}^{\perp}+z+te_{j})\right|\\
=|\hat{x}_{j}-\hat{c}_{j}|\sup_{y\in[c,x]_{J}}\left|\hat{f}_{k}(x_{J}^{\perp}+y)\right|,
\end{multline*}
and in the latter situation we have 
\begin{multline*}
\sup_{y\in[c,x]_{J}}\left|\widehat{\mathfrak{I}_{j}(f)}_{k}(x_{J}^{\perp}+y)\right|\leq|\hat{x}_{j}-\hat{c}_{j}|\sup_{y\in[c,x]_{J}}\sup_{t\in\widehat{[c,x]}_{j}}\left|\hat{f}_{k}(x_{J\cup\{j\}}^{\perp}+y+te_{j})\right|\\
|\hat{x}_{j}-\hat{c}_{j}|\sup_{z\in[c,x]_{J\cup\{j\}}}\left|\hat{f}_{k}(x_{J\cup\{j\}}^{\perp}+z)\right|.
\end{multline*}
Thus, $\mathfrak{I}_{j}(f)$ also satisfies assumption \eqref{eq: int condition}.
Therefore, for any continuous function $f\in\mathcal{C}^{0}(B_{r}(c),H)$
satisfying \eqref{eq: int condition}, and for any finite family $(j_{1},...,j_{m})\in\Lambda^{m}$,
we can consider the function $\mathfrak{I}_{j_{m}}\circ\ldots\circ\mathfrak{I}_{j_{1}}(f)$.

One can ask now whether the equality 
\begin{equation}
\forall j,\,l\in\Lambda:\,\mathfrak{I}_{l}\circ\mathfrak{I}_{j}(f)=\mathfrak{I}_{j}\circ\mathfrak{I}_{l}(f)\label{eq:Fubini}
\end{equation}
holds or not. Indeed, using Fubini's theorem (to the continuous function
$\widehat{[c,x]}_{j}\times\widehat{[c,x]}_{l}\subseteq\R^{2}\ra\R,$
$(t,s)\mapsto g_{x,k}(s,t):=\hat{f}_{k}(x_{\{l\}\cup\{j\}}^{\perp}+te_{j}+se_{l})$)
we obtain 
\begin{align*}
\int_{\hat{c}_{l}}^{\hat{x}_{l}}\widehat{\mathfrak{I}_{j}(f)}_{k}(x_{l}^{\perp}+se_{l})\diff s & =\int_{\hat{c}_{l}}^{\hat{x}_{l}}\int_{\hat{c}_{j}}^{\hat{x}_{j}}\hat{f}_{k}(x_{\{l\}\cup\{j\}}^{\perp}+te_{j}+se_{l})\diff t\diff s.\\
 & =\int_{\hat{c}_{j}}^{\hat{x}_{j}}\int_{\hat{c}_{l}}^{\hat{x}_{l}}\hat{f}_{k}(x_{\{l\}\cup\{j\}}^{\perp}+te_{j}+se_{l})\diff s\diff t\\
 & =\int_{\hat{c}_{j}}^{\hat{x}_{j}}\widehat{\mathfrak{I}_{l}(f)}_{k}(x_{j}^{\perp}+te_{j})\diff t
\end{align*}
which shows that \eqref{eq:Fubini} holds.

Furthermore, for any $j$, $k\in\Lambda$, $x\in H$, we have that
the function $(-\eps,\eps)\ra\R$, ($\eps>0$ is sufficiently small)
$t\mapsto\widehat{\mathfrak{I}_{j}(f)}_{k}(x+te_{j})$ is derivable
and 
\[
\forall\,k\in\Lambda\,\forall x\in B_{r}(c):\ \lim_{t\to0}\frac{\widehat{\mathfrak{I}_{j}(f)}_{k}(x+te_{j})-\widehat{\mathfrak{I}_{j}(f)}_{k}(x)}{t}=f_{k}(x),
\]
which implies that 
\[
\frac{\partial\mathfrak{I}_{j}f}{\partial e_{j}}(x):=\lim_{t\to0}\frac{\mathfrak{I}_{j}f(x+te_{j})-\mathfrak{I}_{j}f(x)}{t}=f(x)\quad\forall j\in\Lambda\,\forall x\in B_{r}(c)
\]
where the limit is with respect to the weak topology.
\begin{rem}
\noindent \label{rem:pathInt}
\begin{enumerate}
\item A possible generalization is to consider continuous functions with
values in another Hilbert space $B$ having an orthogonal Schauder
basis.
\item The case where $\mathbb{F}=\mathbb{C}$ can be treated in a very similar
way, and one can consider the primitive and the derivative with respect
to the real part and the imaginary part of $e_{k}$. Of course, in
this way we do not get complex differentiability but a trivially isomorphic
construction of $\mathcal{D}'(\R^{2})$.
\item In the finite dimensional case $\mathcal{D}'(\R^{n})$, both continuity
and differentiability of a function $f:U\ra\R^{n}$ can be equivalently
formulated considering only the projections $f_{k}:U\ra\R$, as we
did e.g.~in Thm.~\ref{thm:UP of D'}.
\end{enumerate}
\end{rem}

\noindent We can now proceed as in the classical case:
\begin{defn}
\label{def:I^r-Polyn}~
\begin{enumerate}
\item $\mathfrak{I}^{0}f:=f$ and $\mathfrak{I}^{r}:=\mathfrak{I}^{r_{1}}\circ\ldots\circ\mathfrak{I}^{r_{n}}$
for all $r\in\Lambda^{n}$, $n\in\N$;
\item If $B$ is a ball in $H$ and $m\in\Lambda^{n}$, then we define $\theta\in\mathcal{P}_{m}(B)$
if $\theta\in\mathcal{C}^{0}(B,H)$ with $\hat{\theta}_{k}:=\sum_{j\in\Lambda,\,m_{j}\neq0}\hat{\theta}_{kj}$
where $\hat{\theta}_{kj}$ is a polynomial function in $\hat{x}_{j}$
of order $<m_{j}$ \textcolor{black}{whose coefficients are continuous
functions on $B$ independent by $\hat{x}_{j}$.}
\end{enumerate}
\end{defn}

\noindent We can now proceed by following Sebastiao e Silva's idea:
in each ball $B$ in $H$ we define $(r,f)\sim(s,g)$ if there exists
$n\in\N$ such that $r$, $s\in\Lambda^{n}$, $f$, $g\in\mathcal{C}_{\text{p}}^{0}(B)$
and $\mathfrak{I}^{m-r}f-\mathfrak{I}^{m-s}g\in\mathcal{P}_{m}(B)$,
where $m:=\max(r,s)$; $\mathcal{D}_{\text{f}}'(B):=(\bigcup_{n\in\N}\Lambda^{n}\times\mathcal{C}_{\text{p}}^{0}(B))/\sim$;
$\lambda_{B}(f):=[(0,f)]_{\sim}$; $D_{k}([(r,f)]_{\sim}):=\left[(r+e_{k},f)\right]_{\sim}$,
so that $D^{r}f=[(r,f)]_{\sim}\in\mathcal{D}_{\text{f}}'(B)$; finally,
the vector space operations are defined as $D^{r}f+D^{s}g:=D^{m}(\mathfrak{I}^{m-r}f+\mathfrak{I}^{m-s}g)$
($m:=\max(r,s)$) and $\mu\cdot D^{r}f:=D^{r}(\mu f)$ for all $\mu\in\R$;
the restriction to another ball $B'\subseteq B$ is defined by $\left(D^{r}f\right)|_{B'}:=D^{r}(f|_{B'})$.
With these definitions we obtain a separated functor 
\begin{equation}
\mathcal{D}'_{\text{f}}:\mathcal{B}(H)^{\text{op}}\ra\mathbf{Vect}_{\mathbb{F}},\label{eq:functPoset-1}
\end{equation}
where $\mathcal{B}(U)$ is the poset of all the balls contained in
the open set $U\subseteq H$. Using sheafification of this functor,
as explained above in general terms, we obtain a co-universal solution
of this problem.

\section{\label{sec:Colombeau}Co-Universal properties of Colombeau algebras}

A quotient space $A/\!\!\sim$ is the simplest way to get a new space
(in the same category) and a morphism $p:A\ra A/\!\!\sim$ such that
the new notion of equality $a_{1}\sim a_{2}$, for elements $a_{k}\in A$,
implies the standard one: $p(a_{1})=p(a_{2})$. The corresponding
well-known universal property formalizes exactly this idea. Therefore,
whenever we have a quotient space, we can use this general property
to get a first simple characterization of $A/\!\!\sim$ \emph{starting
from the data $A$ and $\sim$}. The defect of this general approach
is clearly that we are not justifying neither the choice of $A$ nor
the equivalence relation $\sim$ as the simplest solution of a clarified
problem. In this section, we first introduce Colombeau special algebra
using the universal property of a quotient, but later, using another
universal property, we make clear why we are using that space and
that equivalence relation.

\subsection{Co-universal property as quotient of moderate nets}

In this section, we want to formulate the co-universal property of
Colombeau algebras by formulating the classical co-universal property
of a quotient at a ``higher level'', i.e.~talking of functors of
$\R$-algebras and natural transformations instead of \textcolor{black}{algebras
and their morphisms.} In the following, we set $I:=(0,1]$, functions
$f\in X^{I}$ are simply called \emph{nets} and denoted as $f=(f_{\eps})$,
any net $\rho=(\rho_{\varepsilon})\in\R_{>0}^{I}$ such that $\rho_{\varepsilon}\to0$
as $\varepsilon\to0^{+}$ will be called a \emph{gauge}, and the set
$AG(\rho^{-1}):=\{(\rho_{\varepsilon}^{-a})\in\R^{I}\mid a\in\R_{>0}\}$
will be called the \emph{asymptotic gauge} generated by $\rho.$ If
$\mathcal{P}\{\eps\}$ is any property of $\eps\in I$, we write $\forall^{0}\eps:\ \mathcal{P}\{\eps\}$
if the property holds for all $\eps$ sufficiently small, i.e.~$\exists\eps_{0}\in I\,\forall\eps\in(0,\eps_{0}]:\ \mathcal{P}\{\eps\}$.
\begin{defn}
\label{def:CA}Let $\Omega\subseteq\R^{d}$ be an open set. The \emph{Colombeau
algebras} is defined by the quotient $\rhoGs(\Omega):=\rhoMod(\Omega)\slash\rhoNeg(\Omega)$,
where 
\[
\rhoMod(\Omega):=\left\{ (u_{\varepsilon})\in\mathcal{C}^{\infty}(\Omega)^{I}\mid\forall K\Subset\Omega\,\forall\alpha\,\exists N\in\N:\ \sup_{x\in K}\left|\partial^{\alpha}u_{\varepsilon}(x)\right|=O\left(\rho_{\varepsilon}^{-N}\right)\right\} 
\]
\[
\rhoNeg(\Omega):=\left\{ (u_{\varepsilon})\in\mathcal{C}^{\infty}(\Omega)^{I}\mid\forall K\Subset\Omega\,\forall\alpha\,\forall n\in\N:\ \sup_{x\in K}\left|\partial^{\alpha}u_{\varepsilon}(x)\right|=O\left(\rho_{\varepsilon}^{n}\right)\right\} 
\]
are resp.~called \emph{moderate} and \emph{negligible} nets ($O(-)$
is the Landau symbol for $\eps\to0^{+}$). The equivalence class defined
by the net $(u_{\varepsilon})\in\rhoMod(\Omega)$ is denoted by $[u_{\varepsilon}]_{\rho}$
or simply by $[u_{\varepsilon}]$ when we are considering only one
gauge.
\end{defn}

\noindent It is easy to prove that $\rhoGs(\Omega)$ is a quotient
$\R$-algebra with pointwise operations $[u_{\eps}]+[v_{\eps}]=[u_{\eps}+v_{\eps}]$
and $[u_{\eps}]\cdot[v_{\eps}]=[u_{\eps}\cdot v_{\eps}]$. Let $\OR$
be the category having as objects open sets $U\subseteq\R^{u}$ of
any dimension $u\in\N=\left\{ 0,1,2,\ldots\right\} $, and smooth
functions as arrows. If we extend $\rhoGs(-)$ on the arrows of $\left(\OR\right)^{\text{op}}$
by $\rhoGs(f)([u_{\eps}]):=[u_{\eps}\circ f]$, we get a functor \textcolor{black}{$\rhoGs(-):\left(\OR\right)^{\text{op}}\ra\mathbf{ALG}_{\mathbb{R}}$,
where $\mathbf{ALG}_{\mathbb{R}}$}\textcolor{olive}{{} }\textcolor{black}{denotes
the category of $\mathbb{{R}}$-algebras.}
\begin{defn}
\textcolor{black}{We denote by $\mathbf{Col}$ the }\textcolor{black}{\emph{category
of Colombeau algebras}}\textcolor{black}{{} and we write $(G,\pi)\in\mathbf{Col}$
if}
\begin{enumerate}
\item \textcolor{black}{$G:\left(\OR\right)^{\text{op}}\ra\mathbf{ALG}_{\mathbb{R}}$
is a functor;}
\item \textcolor{black}{$\pi:\rhoMod(-)\ra G$ is a natural transformation
such that $\rhoNeg(\Omega)\subseteq\text{Ker}(\pi_{\Omega})$ for
all $\Omega\in\OR$. We simply write $\pi_{\Omega}(u_{\eps}):=\pi_{\Omega}((u_{\eps}))$
for all $(u_{\eps})\in\rhoMod(\Omega)$.}
\end{enumerate}
\textcolor{black}{Moreover, we write $(G,\pi)\stackrel{\tau}{\longrightarrow}(F,\alpha)$
in $\mathbf{Col}$ if and only if the following diagram (of natural
transformations) commutes 
\[
\xymatrix{{}\rhoMod(-)\ar[r]^{\alpha}\ar[d]^{\pi} & F\\
G\ar[ru]_{\tau}
}
\]
}

\end{defn}

\begin{thm}
\textcolor{black}{For every $(G,\pi)\in\mathbf{Col}$, there exist
a unique $\tau:(\rhoGs(-),[-])\longrightarrow(G,\pi)$ in $\mathbf{Col}$,
i.e.~$(\rhoGs(-),[-])$ is co-universal in $\mathbf{Col}$, i.e.~is
the simplest way to associate an algebra to any open set $\Omega\subseteq\R^{d}$
and saying that two moderate nets $(u_{\eps})$, $(v_{\eps})\in\rhoMod(\Omega)$
are equal if they differ by a negligible net: $(u_{\eps}-v_{\eps})\in\rhoNeg(\Omega)$.}
\end{thm}

\begin{proof}
\textcolor{black}{We should find $\tau$ such that the following diagram
commutes 
\[
\xymatrix{{}\rhoMod(-)\ar[r]^{\pi}\ar[d]^{[-]} & G\\
{}\rhoGs(-)\ar[ru]_{\tau}
}
\]
The only way $\tau$ can be defined is by setting $\tau_{\Omega}([u_{\varepsilon}]):=\pi_{\Omega}(u_{\varepsilon})$
for all $\Omega\in\OR$. In order to prove that $\tau_{\Omega}$ is
well defined, take two moderate nets $(u_{\varepsilon})$ and $(v_{\varepsilon})$
such that $[u_{\varepsilon}]=[v_{\varepsilon}]$, then we have $\tau_{\Omega}([u_{\varepsilon}])=\pi_{\Omega}(u_{\varepsilon})=\pi_{\Omega}(v_{\varepsilon}+(u_{\varepsilon}-v_{\varepsilon}))=\pi_{\Omega}(v_{\varepsilon})+\pi_{\Omega}(u_{\varepsilon}-v_{\varepsilon})$
because for every $\Omega$, $\pi_{\Omega}$ is an algebra-homomorphism.
Since $\rhoNeg(\Omega)\subseteq\text{Ker}(\pi_{\Omega})$, it follows
that $\tau_{\Omega}([u_{\varepsilon}])=\pi_{\Omega}(u_{\varepsilon})=\pi_{\Omega}(v_{\varepsilon})=\tau_{\Omega}([v_{\varepsilon}])$.}
\end{proof}
Even this simple co-universal property highlights the following possible
generalizations: instead of the category $\OR$ we could take any
category with a notion of smooth function with respect to a ring of
scalars (e.g.~in the field of hyperreals of nonstandard analysis,
see e.g.~\cite{CKKR,Tod09} and references therein; in the ring of
Fermat reals, see \cite{Gio10,Gio10b,Gio11,GioWu16}; in the Levi-Civita
field, see e.g.~\cite{ShBe}, etc. Note that the use of supremum
in Def.~\ref{def:CA} can be easily avoided using an upper bound
inequality, and this can be useful if the ring of scalars is not Dedekind
complete). Instead of the sheaf of smooth functions $\mathcal{C}^{\infty}(-)$,
we can consider any sheaf of smooth functions in more general spaces,
such as diffeological or Frölicher or convenient spaces, see e.g.~\cite{Gio11b}
and references therein. Instead of the asymptotic gauge $AG(\rho^{-1})$,
we can take more general structures, as proved in \cite{GiLu15,GiNi14}.

\subsection{\label{subsec:ColAsQuotientAlg}Co-universal properties as the simplest
quotient algebras}

We now want to show another co-universal property of Colombeau algebras
by completing the idea that \emph{a Colombeau algebra is a quotient
of a subalgebra of $\Coo(-)^{I}$, and moderate and negligible nets
are the simplest choices in order to have non trivial representatives
of zero}. We first define in general what is a quotient subalgebra
of $\Coo(-)^{I}$ as an object of the category $\textsc{QAlg}(\mathcal{C}^{\infty I})$:
\begin{defn}
\label{def:quotientAlgebras} We say that $(G,\pi)$ is a \emph{quotient
subalgebra of }$\Coo(-)^{I}$, and we write $(G,\pi)\in\textsc{QAlg}(\mathcal{C}^{\infty I})$,
if:
\begin{enumerate}
\item $G:\left(\OR\right)^{\text{op}}\ra\mathbf{ALG}_{\R}$ is a functor;
\item \label{enu:piEpi}$\pi:M\ra G$ is a natural transformation such that
$M(\Omega)$ is a subalgebra of $\Coo(\Omega)^{I}$ and $\pi_{\Omega}:M(\Omega)\ra G(\Omega)$
is an epimorphism of $\R$-algebras for all $\Omega\in\OR$.
\end{enumerate}
\end{defn}

\noindent \textcolor{black}{Let us justify why this is related to
quotient algebras. Since for every $\Omega\in\OR$, $\pi_{\Omega}$
is an algebra homomorphism, for any $(u_{\varepsilon})$, $(v_{\varepsilon})\in M(\Omega)\subseteq\mathcal{C}^{\infty}(\Omega)^{I}$
and for any $r\in\mathbb{R}$ we have}
\begin{enumerate}[label=\arabic*)]
\item \textcolor{black}{\label{enu:1}$\pi_{\Omega}(u_{\varepsilon})+\pi_{\Omega}(v_{\varepsilon})=\pi_{\Omega}(u_{\varepsilon}+v_{\varepsilon})$}
\item \textcolor{black}{\label{enu:2}$\pi_{\Omega}(u_{\varepsilon})\cdot\pi_{\Omega}(v_{\varepsilon})=\pi_{\Omega}(u_{\varepsilon}\cdot v_{\varepsilon})$}
\item \textcolor{black}{\label{enu:3}$r\cdot\pi_{\Omega}(u_{\varepsilon})=\pi_{\Omega}(r\cdot u_{\varepsilon})$.}
\end{enumerate}
\textcolor{black}{Moreover, the epimorphism condition Def.~\ref{def:quotientAlgebras}.\ref{enu:piEpi}
means that for every $g\in G(\Omega),$ there exists $(u_{\varepsilon})\in M(\Omega)$
such that $\pi_{\Omega}(u_{\varepsilon})=g.$ This implies 
\begin{equation}
G(\Omega)\simeq M(\Omega)/\text{Ker}(\pi_{\Omega})\text{ in }\mathbf{ALG}_{\R}.\label{eq:isomorphism}
\end{equation}
}

\textcolor{black}{Why are moderate nets $\rhoMod(\Omega)$ the simplest
subalgebra in order to have nontrivial representatives of zero, and
what does this ``nontrivial'' mean? Let $(z_{\varepsilon})\in M(\Omega)$
be such a representative, i.e.~$\pi_{\Omega}(z_{\varepsilon})=0\in G(\Omega)$,
and assume we can take a constant net $\left(J_{\varepsilon}\right)\in M(\Omega)\cap\R^{I}$
such that $\lim_{\varepsilon\longrightarrow0^{+}}|J_{\varepsilon}|=+\infty$.
Then, we have 
\begin{equation}
\pi_{\Omega}(z_{\varepsilon})\cdot\pi_{\Omega}(J_{\varepsilon})=0\cdot\pi_{\Omega}(J_{\varepsilon})=\pi_{\Omega}(z_{\varepsilon}\cdot J_{\varepsilon}),\label{eq:pot infinitesimal}
\end{equation}
and hence also $(z_{\eps}\cdot J_{\eps})$ is another representative
of zero, and this holds for all possible infinite constant nets $(J_{\eps})$.
On the other hand, we would like to have that ``representatives of
zero'' are, in some sense, ``small''. This intuitive idea of being
small is formalized in the following condition:}
\begin{defn}
\label{def:reprZeroInf}We say that \textcolor{black}{\emph{every
representative of zero in}}\textcolor{black}{{} $(G,\pi)$ }\textcolor{black}{\emph{is
infinitesimal}}\textcolor{black}{{} if for all representatives of zero,
$(z_{\varepsilon})\in M(\Omega)\subseteq\mathcal{C}^{\infty}(\Omega)^{I}$
such that $\pi_{\Omega}(z_{\varepsilon})=0\in G(\Omega)$, each compact
set $K\Subset\Omega$ and each multi-index $\alpha\in\N^{d}$, we
have
\begin{equation}
\sup_{x\in K}|\partial^{\alpha}z_{\varepsilon}(x)|:=p_{K,\alpha}(z_{\varepsilon})\to0\text{ as }\varepsilon\to0^{+}.\label{eq:nonTriv}
\end{equation}
}
\end{defn}

\noindent \textcolor{black}{For example, this property does not hold
in nonstandard analysis, see \cite{CKKR}. If this condition holds,
equation \eqref{eq:pot infinitesimal} implies that for each $K\Subset\Omega$
and each multi-index $\alpha\in\N^{d}$, we have $p_{K,\alpha}(z_{\varepsilon}\cdot J_{\varepsilon})=p_{K,\alpha}(z_{\varepsilon})\cdot|J_{\varepsilon}|\longrightarrow0$,
which implies $p_{K,\alpha}(z_{\varepsilon})\leq|J_{\varepsilon}|^{-1}$
for $\eps$ sufficiently small.}

\textcolor{black}{For $\mathcal{R}\subseteq\R^{I}$, let 
\begin{equation}
\infty(\mathcal{R}):=\left\{ (J_{\varepsilon})\in\mathcal{R}\mid\lim_{\varepsilon\longrightarrow0^{+}}|J_{\varepsilon}|=+\infty\right\} \label{eq:infinities}
\end{equation}
be the set of all the infinite nets in $\mathcal{R}$. We then have
two possibilities, which link property \eqref{eq:nonTriv} with the
intuitive idea of trivial representatives of zero:}
\begin{enumerate}
\item \textcolor{black}{$\infty(M(\Omega)\cap\R^{I})$ contains all the
infinite nets. This implies that for all $K$ and for all $\alpha$,
$p_{K,\alpha}(z_{\eps})=0$ for all $\eps$ small (proceed by contradiction
by taking $J_{\eps}:=r\cdot\left|p_{K,\alpha}(z_{\varepsilon})\right|^{-1}$
for all $\eps$ such that $p_{K,\alpha}(z_{\varepsilon})\ne0$ and
where $r\in\R_{>0}$). In this case, the quotient must be trivial
and this situation corresponds to the Schmieden-Laugwitz-Egorov model,
see \cite{ScLa,Ego}.}
\item \textcolor{black}{$\text{\ensuremath{\infty}}(M(\Omega)\cap\R^{I})$
does not contain all the infinite nets.}
\end{enumerate}
\noindent We now define morphisms of $\textsc{QAlg}(\mathcal{C}^{\infty I})$:
\begin{defn}
\label{def:arrowsQuotientAlg} Let $(G,\pi)$, $\left(H,\eta\right)\in\textsc{QAlg}(\Coo(-)^{I})$.
\emph{A morphism of quotient algebras} $i:(G,\pi)\ra(H,\eta)$ is
given by an inclusion 
\begin{equation}
i:\infty(\pi_{\Omega})\hookrightarrow\infty(\eta_{\Omega}),\quad\forall\Omega\in\OR
\end{equation}
where $\infty(\pi_{\Omega}):=\infty(M(\Omega)\cap\R^{I})$ for all
$\Omega\in\OR$.
\end{defn}

\noindent We have the following
\begin{lem}
Quotient algebras of $\Coo(-)^{I}$ and their morphisms form a category
$\textsc{QAlg}(\mathcal{C}^{\infty I})$.
\end{lem}

Therefore, a co-universal quotient algebra $(G,\pi)$ (when it exists)
has the smallest class of infinities. We will see that it necessarily
follows it has the largest kernel as well.

In the following theorem, we use the notation $[-]_{\Omega}:\,(x_{\eps})\in{}\rhoMod(\Omega)\longmapsto[x_{\eps}]_{\Omega}\in{}\rhoGs(\Omega)$
for all $\Omega\in\OR$.
\begin{thm}
\label{thm: 2nd up of CGS}Assume that:
\begin{enumerate}
\item \label{enu:1st assump}$(G,\pi)\in\textsc{QAlg}(\Coo{}^{I})$ is a
quotient algebra;
\item \label{enu:2nd assump}Every representative of $0$ in $(G,\pi)$
is \textcolor{black}{infinitesimal, i.e.~Def.~\ref{def:reprZeroInf}
holds;}
\item \label{enu:6th assump}If $(u_{\eps})\in M(\R)\cap\R^{I}$ then $\exists(v_{\eps})\in\infty(\pi_{\R})\,\forall^{0}\eps:\,|u_{\eps}|\leq v_{\eps}$
(constant nets are bounded by infinities).
\end{enumerate}
For all open set $\Omega\subseteq\R^{n}$, we also assume that:
\begin{enumerate}[resume]
\item \label{enu:3rd assump}$(\rho_{\eps}^{-1})\in M(\Omega)$;
\item \label{enu:4th  assump}$\forall(u_{\eps})\in M(\Omega)\,\forall K\Subset\Omega\,\forall\alpha\in\N^{n}:\ p_{K\alpha}(u_{\eps})\in M(\R)\cap\R^{I}$;
\item \label{enu:5th assump}\textcolor{black}{Let $(u_{\varepsilon})\in\Coo(\Omega)^{I}$.
If for any $K\Subset\Omega$, $\alpha\in\N^{n}$ there exists $(v_{\varepsilon})\in\text{\ensuremath{\infty}}(M(\Omega)\cap\R^{I})$
such that $\forall^{0}\varepsilon:\ p_{K,\alpha}(u_{\varepsilon})\leq v_{\varepsilon}$,
then $(u_{\varepsilon})\in M(\Omega)$} (infinities determine $M(\Omega)$);
\item \label{enu:7th assump}$(G,\pi)$ is co-universal among all the quotient
algebras verifying the previous conditions.
\end{enumerate}
Then $G(\Omega)\simeq{}\rhoGs(\Omega)$ as $\R$-algebras, i.e.~in
the category $\mathbf{ALG}_{\R}$. Moreover, $\infty(\pi_{\R})=\infty([-]_{\R})$,
$\text{\emph{Ker}}(\pi_{\Omega})=\text{\emph{Ker}}([-]_{\Omega})$.
\end{thm}

\begin{proof}
Conditions \ref{enu:6th assump}, \ref{enu:4th  assump}, \ref{enu:5th assump}
are equivalent to 
\begin{align*}
M(\Omega) & =\mathcal{E}_{\text{M}}(\infty(\pi_{\R}),\Omega)\\
 & :=\left\{ (u_{\eps})\in\mathcal{C}^{\infty}(\Omega)^{I}\mid\forall K,\,\alpha\,\exists(v_{\eps})\in\infty(\pi_{\R})\,\forall^{0}\eps:\,p_{K\alpha}(u_{\eps})\leq v_{\eps}\right\} .
\end{align*}
In fact assumptions \ref{enu:6th assump}, \ref{enu:4th  assump}
yield $M(\Omega)\subseteq\mathcal{E}_{\text{M}}(\infty(\pi_{\R}),\Omega)$,
whereas \ref{enu:5th assump} gives the opposite inclusion.

The Colombeau algebra ${}\rhoGs(\Omega)$ satisfies conditions \ref{enu:1st assump}-\ref{enu:5th assump}.
Now, we prove that it also satisfies condition \ref{enu:7th assump}.
Let $(G,\pi)$ be another quotient algebra satisfying conditions \ref{enu:1st assump}-\ref{enu:5th assump},
and take $(x_{\eps})\in\infty([-]_{\R})$, so that $(x_{\eps})$ is
an infinite but moderate net: 
\[
\exists N\in\N\,\forall^{0}\eps:\,|x_{\eps}|\leq\rho_{\eps}^{-N}.
\]
From assumption \ref{enu:3rd assump}, we have that $(\rho_{\eps}^{-1})\in M(\Omega)$,
and hence $(\rho_{\eps}^{-N})$ as well, since $M(\Omega)$ is a subalgebra
of $\mathcal{C}^{\infty}(\Omega)^{I}$. Thereby $(\rho_{\eps}^{-N})\in\infty(\pi_{\R})$.
By using condition \ref{enu:5th assump} with $u_{\eps}(-)\equiv x_{\eps}$
and $v_{\eps}(-)\equiv\rho_{\eps}^{-N}$, we get that $(x_{\eps})\in\infty(\pi_{\R})$.
This shows that $\infty([-]_{\R})\subseteq\infty(\pi_{\R})$, i.e.~${}\rhoGs(\Omega)$
is co-universal. Note that this only implies (from Thm.~\ref{thm:UPuniqueUpToIso})
that $G\simeq\rhoGs$ in $\textsc{QAlg}(\Coo{}^{I})$, which is not
our final claim. Moreover, we never used condition \ref{enu:2nd assump}
so far.

We now prove that $\text{Ker}(\pi_{\Omega})\subseteq\text{Ker}([-]_{\Omega})$:
Let $\pi_{\Omega}(z_{\eps})=0$. We have already seen that this and
\ref{enu:2nd assump} imply 
\begin{equation}
\forall J_{\eps}\in\infty(\pi_{\R})\,\forall K,\alpha\,\forall^{0}\eps:\,p_{K\alpha}(z_{\eps})\leq J_{\eps}^{-1}.\label{eq:4.6-1}
\end{equation}
Since we have already proved that $\infty([-]_{\R})\subseteq\infty(\pi_{\R})$,
\eqref{eq:4.6-1} yields $[z_{\eps}]_{\R}=0$. Therefore, if $(G,\pi)$
is also a co-universal solution, formula \eqref{eq:isomorphism} gives
\begin{multline*}
G(\Omega)\simeq M(\Omega)/\text{Ker}(\pi_{\Omega})=\mathcal{E}_{\text{M}}(\infty(\pi_{\R}),\Omega)/\text{Ker}(\pi_{\Omega})\\
=\mathcal{E}_{\text{M}}([-]_{\R},\Omega)/\text{Ker}([-]_{\Omega})={}\rhoGs(\Omega).
\end{multline*}
\end{proof}

\subsubsection{\label{subsec:A-particular-case}A particular case: co-universal
property of Robinson-Colombeau generalized numbers}

Proceeding as in Sec.~\ref{subsec:ColAsQuotientAlg}, we obtain a
co-universal property of the ring of Robinson-Colombeau generalized
numbers, i.e.~the ring of scalars of Colombeau theory with an arbitrary
gauge $\rho$.
\begin{defn}
\label{def:RCring}~
\begin{enumerate}
\item $\R_{\rho}:=\left\{ (x_{\varepsilon})\in\R^{I}\mid\exists N\in\N:\;x_{\varepsilon}=O(\rho_{\varepsilon}^{-N})\;\text{as \ensuremath{\varepsilon\to0^{+}}}\right\} $
is called the set of $\rho$-\emph{moderate nets} of numbers.
\item Let $(x_{\varepsilon})$, $(y_{\varepsilon})\in\R_{\rho}$. We write
$(x_{\varepsilon})\sim_{\rho}(y_{\varepsilon})$ if $\forall n\in\N:\;x_{\varepsilon}-y_{\varepsilon}=O(\rho_{\varepsilon}^{-n})\text{ as \ensuremath{\varepsilon\to0^{+}}}$.
It is easy to prove that $\sim_{\rho}$ is a congruence relation on
the ring $\R_{\rho}$ of moderate nets with respect to pointwise operations.
\item The \emph{Robinson-Colombeau ring} \emph{of generalized numbers} is
defined as $\rti:=\R_{\rho}\slash\sim_{\rho}$. The equivalence class
defined by $(x_{\varepsilon})\in\R_{\rho}$ is simply denoted as $[x_{\varepsilon}]\in\rti$.
\end{enumerate}
\end{defn}

\noindent Clearly, the ring of Robinson-Colombeau is isomorphic to
the subring of Colombeau generalized functions $f\in\rhoGs(\R)$ whose
derivative is zero $f'=[f'_{\eps}]=0$, see e.g.~\cite{GKOS}.

Similarly to the category $\textsc{QAlg}(\Coo{}^{I})$ we can now
introduce the category of quotient subrings of $\R^{I}$:
\begin{defn}
We say that $(G,\pi)$ is a \emph{quotient subring of }$\mathbb{R}^{I}$,
and we write $(G,\pi)\in\textsc{QRing}(\R^{I})$, if
\begin{enumerate}
\item $G$ is a ring;
\item $\pi:\mathcal{R}\ra G$ is an epimorphism of rings, where the domain
$\mathcal{R}\subseteq\R^{I}$ is a subring of $\R^{I}$.
\end{enumerate}
Let $\left(H,\eta\right)\in\textsc{QRing}(\R^{I})$. Then a \emph{morphism
of quotient rings} $i:(G,\pi)\longrightarrow(H,\eta)$ is given by
an inclusion
\begin{equation}
i:\infty(\pi)\hookrightarrow\infty(\eta),
\end{equation}
where $\infty(\pi):=\infty(\mathcal{R})$. Similarly to Def.~\ref{def:reprZeroInf},
we say that \emph{every representative of zero in $(G,\pi)$ is infinitesimal}
if for all representatives of zero, i.e.~$(z_{\eps})\in\mathcal{R}$
such that $\pi(z_{\eps})=0\in G$, we have $\lim_{\eps\to0^{+}}z_{\eps}=0$.

The ring $\rti$ of Robinson-Colombeau is, up to isomorphisms of rings,
the simplest quotient ring where every representative of zero is infinitesimal.
This implies to have the smallest class of infinities, and consequently,
the largest kernel. The proof is simply a particular case of that
of Thm.~\ref{thm: 2nd up of CGS}:
\end{defn}

\begin{thm}
\label{thm:RCring}Assume that:
\begin{enumerate}
\item \label{enu:1st assump-2}$(G,\pi)$ is a \emph{quotient subring of
}$\mathbb{R}^{I}$;
\item \label{enu:2nd assump-2}Every representative of $0$ in $(G,\pi)$
is \textcolor{black}{infinitesimal;}
\item \label{enu:5th assump-2} If $(x_{\varepsilon})\in\mathcal{R}$, then
$\exists(v_{\varepsilon})\in\text{\ensuremath{\infty}}(\mathcal{R})\,\forall^{0}\varepsilon:\ |x_{\varepsilon}|\leq v_{\varepsilon}$
(nets are bounded by infinities);
\item \label{enu:3rd assump-2}$(\rho_{\eps}^{-1})\in\mathcal{R}$;
\item \label{enu:4th  assump-2}Let $(x_{\varepsilon})\in\R^{I}$. If there
exist $(v_{\varepsilon})\in\text{\ensuremath{\infty}}(\mathcal{R})$
such that $\forall^{0}\varepsilon:\ |u_{\varepsilon}|\leq v_{\varepsilon}$,
then $(u_{\varepsilon})\in\mathcal{R}$ (infinities determine $\mathcal{R}$).
\item $(G,\pi)$ is co-universal among all the quotient rings satisfying
the previous conditions.
\end{enumerate}
Then, $G\simeq\rti$ as rings. Moreover, $\infty(\mathcal{R})=\infty(\R_{\rho})$
is the smallest class of infinities and $\text{\emph{Ker}\ensuremath{(\pi)=\text{\emph{Ker}}}}([-])$
is the largest kernel.
\end{thm}

\noindent See \cite{Tod09,Tod13} for a characterization up to isomorphisms
of the \emph{field} of scalars one has in the nonstandard approach
to Colombeau theory.

\section{\label{sec:GSF}Universal property of spaces of generalized smooth
functions}

Generalized smooth functions (GSF) are the simplest way to deal with
a very large class of generalized functions and singular problems,
by working directly with all their $\rho$-moderate smooth regularizations.
GSF are close to the historically original conception of generalized
function, \cite{Dir26,Lau89,Kat-Tal12}: in essence, the idea of authors
such as Dirac, Cauchy, Poisson, Kirchhoff, Helmholtz, Kelvin and Heaviside
(who informally worked with ``numbers'' which also comprise infinitesimals
and infinite scalars) was to view generalized functions as certain
types of smooth set-theoretical maps obtained from ordinary smooth
maps by introducing a dependence on suitable infinitesimal or infinite
parameters. For example, the density of a Cauchy-Lorentz distribution
with an infinitesimal scale parameter was used by Cauchy to obtain
classical properties which nowadays are attributed to the Dirac delta,
\cite{Kat-Tal12}. More generally, in the GSF approach, generalized
functions are seen as set-theoretical functions defined on, and attaining
values in, the non-Archimedean ring of scalars $\rti$. The calculus
of GSF is closely related to classical analysis sharing several properties
of ordinary smooth functions. On the other hand, GSF include all Colombeau
generalized functions and hence also all Schwartz distributions \cite{GKV15,GKOS,GKV24}.
They allow nonlinear operations on generalized functions and unrestricted
composition \cite{GKV15,GKV24}. They enable to prove a number of
analogues of theorems of classical analysis for generalized functions:
e.g., mean value theorem, intermediate value theorem, extreme value
theorem, Taylor’s theorems, local and global inverse function theorems,
integrals via primitives, and multidimensional integrals \textcolor{red}{\cite{Gio-Kun18,Gio-Kun17,GKV24}}.
With GSF we can develop calculus of variations and optimal control
for generalized functions, with applications e.g.~in collision mechanics,
singular optics, quantum mechanics and general relativity, see \cite{LLG,Gastao2022}
and \cite{GaVe,KKO:08} for a comparison with CGF. We have new existence
results for nonlinear singular ODE and PDE (e.g.~a Picard-Lindelöf
theorem for PDE), \cite{Lu-Gi16,GiLu}, and with the notion of \emph{hyperfinite
Fourier transform} we can consider the Fourier transform of any GSF,
without restriction to tempered type, \cite{MTG}. GSF with their
particular sheaf property define a Grothendieck topos, \cite{GKV24}.
\begin{defn}
\label{def:GSF}Let $X\subseteq\rti^{n}$ and $Y\subseteq\rti^{d}$.
We say that $f:X\rightarrow Y$ is a GSF ($f\in\gsf(X,Y)$), if
\begin{enumerate}
\item $f:X\rightarrow Y$ is a set-theoretical function
\item There exists a net $(f_{\eps})\in\mathcal{C}^{\infty}(\R^{n},\R^{d})$
such that for all $[x_{\eps}]\in X$ (i.e.~for all representatives
$(x_{\eps})$ of any point $x=[x_{\eps}]\in X$):
\begin{enumerate}
\item $f(x)=[f_{\eps}(x_{\eps})]$ (we say that $f$ \emph{is defined by
the net }$(f_{\eps})$)
\item $\forall\alpha\in\N^{n}:\ \left(\partial^{\alpha}f_{\eps}(x_{\eps})\right)$
is $\rho$-moderate.
\end{enumerate}
\end{enumerate}
\end{defn}

Following \cite[Def.~3.1]{Ver08} it is possible to give an equivalent
definition of GSF as a quotient set: therefore a co-universal characterization
similar to the previous ones given for CGF (Sec.~\ref{sec:Colombeau})
is possible. However, in the present section we want to present a
universal property of spaces of GSF as the simplest way to have set-theoretical
functions defined on generalized numbers and having arbitrary derivatives.
As we will see below, this property is important because it formalizes
the idea that GSF contains all the possible $\rho$-moderate regularizations,
e.g.~as obtained by convolution with a mollifier of the form $\frac{1}{\rho_{\eps}}\mu\left(\frac{x}{\rho_{\eps}}\right)$,
see e.g.~\cite{GKV24}.

The ring of scalars $\rti$ is hence the basic building block in the
definition of GSF. Using the results of Sec.~\ref{subsec:A-particular-case},
in this section we could use any co-universal solution of Thm.~\ref{thm:RCring},
but that would only result into a useless abstract language, so that
we work directly with $\rti$, as defined in Def.~\ref{def:RCring}.

First of all, derivatives of GSF are well-defined on so-called sharply
open sets: Let $x$, $y\in\rti$, we write $x\leq y$ if for all representative
$[x_{\eps}]=x$, there exists $[y_{\eps}]=y$ such that $\forall^{0}\eps:\ x_{\eps}\leq y_{\eps}$;
on $\rti^{n}$, we consider the natural extension of the Euclidean
norm, i.e. $|[x_{\eps}]|:=[|x_{\eps}|]\in\rti$. Even if this generalized
norm takes value in $\rti$, it shares essential properties with classical
norms, like the triangle inequality and absolute homogeneity. It is
therefore natural to consider on $\rti^{n}$ the topology generated
by balls $B_{r}(x):=\left\{ y\in\rti\mid|x-y|<r\right\} $, for $r\in\rti_{\ge0}$
and invertible, which is called \textit{sharp topology}, and its elements
\emph{sharply open sets}.
\begin{thm}
\label{thm:derGSF}Let $U\subseteq\rcrho^{n}$ be a sharply open set
and $\alpha\in\N^{n}$, then the map given by
\[
\partial^{\alpha}:[f_{\eps}(-)]\in\gsf(U,\rti^{d})\mapsto[\partial^{\alpha}f_{\eps}(-)]\in\gsf(U,\rti^{d})
\]
is well-defined, i.e.~it does not depend on the net of smooth functions
$(f_{\eps})$ that defines the GSF $[f_{\eps}(-)]:x=[x_{\eps}]\in U\mapsto[f_{\eps}(x_{\eps})]\in\rti^{d}$.
\end{thm}

For any sharply open set $U\subseteq\rti^{n}$, we set 
\[
\mathcal{M}_{U}^{d}:=\left\{ (f_{\eps})\in\Coo(\R^{n},\R^{d})^{I}\mid\forall[x_{\eps}]\in U\,\forall\alpha\in\N^{n}:\ \left(\partial^{\alpha}f_{\eps}(x_{\eps})\right)\in\R_{\rho}\right\} .
\]
We hence have a map $[-]_{\text{f}}$ that allows us to construct
GSF starting from nets of smooth functions: 
\[
[-]_{\text{f}}:(f_{\eps})\in\mathcal{M}_{U}^{d}\mapsto[f_{\eps}(-)]\in\gsf(U,\rti^{d}).
\]
By Def.~\ref{def:GSF} of GSF, this map is onto. Thanks to Thm.~\ref{thm:derGSF},
derivatives $\partial^{\alpha}:\gsf(U,\rti^{d})\ra\gsf(U,\rti^{d})\subseteq\Set(U,\rti^{d})$
are $\eps$-wise well-defined, i.e.~using the ring epimorphism $[-]:\R_{\rho}\ra\rti$,
the map $\pi^{\alpha}(f_{\eps}):[x_{\eps}]\in U\mapsto[\partial^{\alpha}f_{\eps}(x_{\eps})]\in\rti^{d}$
is defined for all nets $(f_{\eps})\in\mathcal{M}_{U}^{d}$ and makes
this diagram commute\textcolor{black}{
\[
\xymatrix{\gsf(U,\rti^{d})\ar[r]^{\partial^{\alpha}} & \Set(U,\rti^{d})\\
\mathcal{M}_{U}^{d}\ar[u]^{[-]_{\text{f}}}\ar[ru]_{\pi^{\alpha}}
}
\]
The space $\gsf(U,\rti^{d})$ and the maps $\partial^{\alpha}$, $[-]_{\text{f}}$
are the simplest way to make this diagram commute, i.e.~we have the
following}
\begin{thm}
\label{thm:GSFUniv}Let $U\subseteq\rcrho^{n}$ be a sharply open
set and $d\in\N$. If $G\in\Set$, $q$ and $(D^{\alpha})_{\alpha\in\N^{n}}$
are such that $D^{0}:G\hookrightarrow\Set(U,\rti^{d})$ is the inclusion
and for all $\alpha\in\N^{n}$:\textcolor{black}{
\[
\xymatrix{G\ar[r]^{D^{\alpha}\ \ \ \ \ \ } & \Set(U,\rti^{d})\\
\mathcal{M}_{U}^{d}\ar[u]^{q}\ar[ru]_{\pi^{\alpha}}
}
\]
where $q$ is surjective, then there exists one and only one $\phi:G\ra\gsf(U,\rti^{d})$
such that
\[
\xymatrix{G\ar[r]^{\phi\ \ \ \ \ \ \ \ } & \gsf(U,\rti^{d})\\
\mathcal{M}_{U}^{d}\ar[u]^{q}\ar[ru]_{[-]_{\text{f}}}
}
\]
and which preserves derivatives, i.e.~$\partial^{\alpha}\phi(F)=D^{\alpha}F$
for all $F\in G$ and all $\alpha\in\N^{n}$.}
\end{thm}

\begin{proof}
Since $q$ is surjective, if $F\in G$, we can find $(f_{\eps})\in\mathcal{M}_{U}^{d}$
such that $F=q(f_{\eps})$. We necessarily have to define $\phi(F):=\phi(q(f_{\eps}))=[f_{\eps}]_{\text{f}}=[f_{\eps}(-)]$.
The map $\phi$ is well-defined: if $F=q(\bar{f}_{\eps})$, then $D^{0}(q(\bar{f}_{\eps}))=q(\bar{f}_{\eps})=\pi^{0}(\bar{f}_{\eps})=[f_{\eps}]_{\text{f}}$.
It remains only to prove the preservation of derivatives. We have
$\partial^{\alpha}\phi(F)=\partial^{\alpha}[f_{\eps}(-)]=[\partial^{\alpha}f_{\eps}(-)]=[\partial^{\alpha}f_{\eps}]_{\text{f}}$,
and $D^{\alpha}F=D^{\alpha}\left(q(f_{\eps})\right)=\pi^{\alpha}(f_{\eps})=[\partial^{\alpha}f_{\eps}]_{\text{f}}$.
\end{proof}
We close this section by noting that trivially $\Set(U,\rti^{d})$
is not a universal solution of the same problem because we cannot
have a surjection like $[-]_{\text{f}}$. Finally, the universal solution
$\gsf(U,\rti^{d})$ is in a certain sense minimal because it is possible
to prove that $\gsf(\widetilde{\Omega}_{c},\rti^{d})\simeq\rhoGs(\Omega)$,
where $\Omega\subseteq\R^{n}$ is any open set and $\widetilde{\Omega}_{c}:=\left\{ [x_{\eps}]\in\rti^{n}\mid\exists K\Subset\Omega\,\forall^{0}\eps:\ x_{\eps}\in K\right\} $,
see e.g.~\cite{GKV24,GKV15}, i.e.~up to isomorphism we get exactly
the Colombeau algebra on $\Omega$.

\section{Conclusions}

The aim of the present paper is not to compare, in some sense, different
spaces of generalized functions, but to characterize them using suitable
universal properties. However, we want to briefly close the paper
trying to clarify several frequent misunderstandings concerning Colombeau
theory and nonlinear operations on distributions.

If we are interested only to linear operations, we showed in Sec.~\ref{sec:distributions}
that the space of Schwartz distributions is the simplest solutions,
and any other solution would be less optimal. Schwartz impossibility
theorem, see e.g.~\cite{GKOS} and references therein, states that
nonlinear operations are seriously problematic for distributions.

In Sec.~\ref{sec:Colombeau}, we presented Colombeau theory as the
simplest solution of this problem among quotient algebras. A common
mainstream objection to Colombeau construction is that distributions
are not intrinsically embedded in the corresponding algebra. Besides
the fact that this is false, because a different index set instead
of $I=(0,1]$ would enable to have the searched intrinsic embedding
(and essentially with the same notations and with the same basic ideas,
see \cite{GiLu15}; see also \cite{GKOS}), one could argue that if
Colombeau algebra had historically appeared before Schwartz distributions,
now some people would not accept the latter because they do not intrinsically
embed into the former. On the contrary, Colombeau theory is the formalization
of the method of regularizations: a convenient setting, sharing several
properties with ordinary smooth functions, and containing convolutions
with any mollifier of the form $\frac{1}{\rho_{\eps}}\mu\left(\frac{x}{\rho_{\eps}}\right)$.
We motivate the universal properties of GSF exactly in this way.

The real technical drawbacks of Colombeau algebras are the lacking
of closure with respect to composition (see e.g.~\cite{GKOS}), not
good properties of Fourier transform as well as multidimensional integration
on infinite sets (see e.g.~\cite{MTG}), and the lacking of general
existence results for differential equations, such as the Picard-Lindelöf
or the Nash-Moser theorems. GSF solve almost all these problems.

On the other hand, an important open problem of GSF is a clear link
between the natural notion of pointwise solution for GSF (i.e.~regularized)
differential equations, and the notion of weak solution in Sobolev
spaces.

In the present paper, we showed that any other idea to formalize the
method of regularizing a singular problem would necessarily be less
simple than Colombeau algebras or spaces of GSF.

\end{document}